\documentclass[12pt,a4paper]{amsart}
\usepackage{graphicx}
\usepackage{amsmath}
\usepackage{tikz-cd} 
\usepackage{amsthm}
\usepackage{amssymb}
\usepackage{faktor}
\usepackage{color,soul}
\usepackage{mathtools}
\usepackage{mathrsfs}
\usepackage {latexsym}
\usepackage{hyperref}
\usepackage{epsfig}
\usepackage{tikz}
\usetikzlibrary{arrows}
\usepackage{tikzit}
\usepackage{indentfirst}

\makeatletter
\@namedef{subjclassname@2020}{\textup{2020} Mathematics Subject Classification}
\makeatother

\hypersetup{
    colorlinks=true,       
    linkcolor=blue,        
    citecolor=blue,        
    urlcolor=blue,         
}

\newtheorem{theorem}{{Theorem}}[section]
\newtheorem{lemma}[theorem]{{Lemma}}

\newtheorem{corollary}[theorem]{{Corollary}}
\theoremstyle{definition}
\newtheorem{definition}[theorem]{{Definition}}

\newtheorem{remark}[theorem]{Remark}

\newtheorem*{ack}{Acknowledgement}
\numberwithin{equation}{subsection}

\tikzstyle{new style 0}=[fill=white, draw=black, shape=circle]

\tikzstyle{new edge style 0}=[<-]
\tikzstyle{new edge style 1}=[->]
\tikzstyle{new edge style 2}=[<->]
\tikzstyle{new edge style 3}=[-, fill=none, draw={rgb,255: red,26; green,118; blue,230}]
\tikzstyle{new edge style 4}=[-, draw={rgb,255: red,226; green,8; blue,74}]
\tikzstyle{new edge style 5}=[->, fill=none, draw={rgb,255: red,10; green,158; blue,238}]

\usepackage[all,cmtip]{xy}

\newcommand{\degree}{\operatorname{deg}}
\newcommand{\res}{\operatorname{res}}

\newcommand{\modulo}{\operatorname{mod}}
\newcommand{\Id}{\operatorname{Id}}

\newcommand{\Spe}{\operatorname{Spe}}

\newcommand{\diag}{\operatorname{Diag}}
\newcommand{\Aut}{\operatorname{Aut}}

\newcommand{\Inn}{\operatorname{Inn}}
\newcommand{\Out}{\operatorname{Out}}
\newcommand{\T}{\operatorname{T}}

\newcommand{\R}{\operatorname{R}}

\newcommand{\s}{\operatorname{S}}
\newcommand{\sw}{\operatorname{\mathcal{SW}}}

\newcommand{\Z}{\operatorname{Z}}
\newcommand{\p}{\operatorname{P}}
\newcommand{\C}{\operatorname{C}}
\newcommand{\U}{\operatorname{U}}

\newcommand{\odd}{\mathrm{odd}}
\newcommand{\even}{\mathrm{even}}

\newcommand{\Q}{\operatorname{Q}}
\usepackage[hmargin=2.5cm,vmargin=2.5cm]{geometry}

\title{Symmetry and Rigidity of Star-Shaped Coxeter Systems}

\author{Arijit Mahato} 

\author{Tushar Kanta Naik}

\author{A Rameswar Patro}

\subjclass[2020]{Primary 20F55; Secondary  20F28, 20E36}
\keywords{Diagram Twisting, Diagram Automorphisms, Isomorphism Problem, Partial Conjugation, Pseudo-transposition, Rigidity, $R_\infty$-Property, Transvections}

\begin{document}

\begin{abstract}
We provide a complete description of the automorphism group $\Aut (W)$ of a Coxeter group $W$ admitting a star-shaped finite Coxeter diagram. We prove that each automorphism decomposes as a product of inner and diagram automorphisms, along with three additional types: transvections and two families of partial conjugations. Furthermore, we investigate the natural short exact sequence $1 \to \Inn (W) \to \Aut (W) \to \Out (W) \to 1$. Using Moussong's criteria for hyperbolicity, we show that these groups possess the $R_\infty$-property. Finally, we establish rigidity properties for these groups using known techniques and provide a solution to the isomorphism problem within the class of star-shaped Coxeter systems. 
\end{abstract}

\maketitle
\section{Introduction}
A group $W$ is called a {\it Coxeter group} if it admits a presentation $W = \langle S \mid (s_i s_j)^{m(s_i, s_j)} = 1, s_i, s_j\in S \rangle$, where $m(s_i, s_j)= m(s_j, s_i) \in \mathbb{N} \cup \{\infty\}$, and $m(s_i, s_j)= 1$ if and only if $s_i=s_j$. The pair $(W, S)$ is called a {\it Coxeter system}, and the values $m(s_i, s_j)$ for $s_i\neq s_j$ are known as the {\it labels} (or {\it exponents}) of the system. The cardinality of $S$ is defined as the {\it rank} of the system. For each subset $I \subseteq S$, the subgroup $W_I = \langle I \rangle$ is called a {\it standard parabolic subgroup}, and its conjugates are called {\it parabolic subgroups}, which are themselves Coxeter groups.

A convenient way to encapsulate the presentation of a Coxeter group is via graphs or diagrams. The {\it Coxeter graph} $\Gamma_{(W, S)}$ has vertex set $S$, where two distinct vertices $s_i, s_j$ are joined by an edge if and only if $m(s_i,s_j) \ge 3$; with labels for $m(s_{i},s_{j}) \ge 4$. Parallel to this is the {\it finite (Coxeter) diagram} $\mathcal{V}_{(W, S)}$, which shares the same vertex set $S$. In this diagram, two distinct vertices $s_i, s_j \in S$ are joined by an edge labelled with $m(s_i, s_j)$ if and only if $m(s_i, s_j)$ is finite. These graphical representations reveal structural features at a glance. For instance, if the  Coxeter graph $\Gamma_{(W, S)}$ is disconnected, then $W$ decomposes as a direct product of the parabolic subgroups associated with the connected components. When $\Gamma_{(W, S)}$ is connected, $W$ cannot be expressed as a direct product of proper parabolic subgroups; accordingly, the Coxeter system $(W, S)$ is termed {\it irreducible}. On the other hand, if the finite diagram $\mathcal{V}_{(W, S)}$ is disconnected, then  $W$ decomposes as a free product of the parabolic subgroups. It is worth noting that individually, these two diagrams determine the underlying Coxeter group uniquely. Consequently, any graph automorphism of these diagrams extends naturally to an automorphism of the group $W$, referred to as a {\it diagram automorphism}. The collection of all diagram automorphisms of $W$ forms a subgroup of $\Aut(W)$, denoted by $\diag(W)$.

While a Coxeter system $(W, S)$ uniquely determines the group $W$, the group itself does not uniquely determine the system. This motivates the well-known isomorphism problem for Coxeter groups: {\it determining when two distinct Coxeter presentations (or equivalently Coxeter graphs or finite diagrams) define isomorphic groups.} A classic example is the dihedral group $D_{12}$, which attains two non-isomorphic Coxeter presentations: $\langle r,s \mid r^2=s^2=(rs)^6=1\rangle$ (irreducible), and $\langle a,b \mid a^2=b^2=(ab)^3=1\rangle \times \langle c \mid c^2=1\rangle$ (reducible). This raises the question, posed by Cohen \cite[Problem 6.5]{Coh-91}, {\it whether two non-isomorphic irreducible Coxeter graphs always correspond to non-isomorphic Coxeter groups}. M\"uhlherr \cite{Mu-00} refuted this by constructing a Coxeter group that admits two non-isomorphic irreducible Coxeter systems. This counterexample was achieved through a general geometric operation later formalized as diagram twisting \cite{BMMN-02}.

A Coxeter group $W$ is called {\it rigid} if all Coxeter systems of $W$ induce the same Coxeter graph up to isomorphism. It is called  {\it strongly rigid} if any two Coxeter generating sets are conjugate. The system $(W, S)$ is called {\it reflection rigid}, if for every Coxeter generating set $R \subset S^W$, there exists an automorphism $\alpha \in \Aut (W)$ such that $\alpha(R) = S$. Finally, the system $(W, S)$ is called {\it strongly reflection rigid}, if for any Coxeter generating set $R \subset S^W$, there exists an element $w \in W$ such that $R = wSw^{-1}$.

A Coxeter system $(W, S)$ is said to be {\it even} (respectively {\it odd}) if all the finite labels are even (respectively odd). Radcliffe \cite{Ra-01} showed that {\it strongly even} Coxeter systems (where all the finite labels are $2$ and/or multiples of $4$) are rigid. More generally, Bahls \cite{Ba-02} proved that a Coxeter group cannot admit two non-isomorphic even Coxeter systems. He \cite{Ba-04} also gave necessary and sufficient conditions for the strong rigidity of an even Coxeter group. {\it Large-type} Coxeter groups (where no finite labels are $2$) are reflection rigid up to diagram twisting \cite{MW-02}. Infinite irreducible Coxeter groups that do not contain the infinite dihedral group as a parabolic subgroup are strongly rigid \cite{CMu-07, FHM-06}.  Charney and Davis \cite{CD-00} established a significant geometric criterion for a strongly rigid Coxeter group: if a Coxeter group $W$ acts effectively, properly, and cocompactly on some contractible manifold, then $W$ is strongly rigid. For general Coxeter groups, an equivalent criterion for strong rigidity and strong reflection rigidity is given in \cite{HMN-18}.

The isomorphism problem and the study of automorphisms are deeply intertwined. While the former seeks the conditions under which different Coxeter systems yield the same abstract group, the latter investigates the internal symmetries of these groups. This conceptual bridge naturally leads to a systematic study of automorphisms of Coxeter groups. 

A Coxeter group is called {\it right-angled} (in short {\it RACG}) if all the finite labels are $2$. If there are no finite labels, it is referred to as a {\it universal} Coxeter group. A universal Coxeter group of rank $n$ is denoted by $\U_n$. Let $W$ be an RACG with $\Gamma$ as its Coxeter graph. Tits \cite{Ti-88} showed that $\Aut(W)=\Spe(W)\rtimes \Aut(F(\Gamma))$, where $\Spe(W)$ is the group of automorphisms of $W$ that preserves the conjugacy classes of involutions of $W$, and $F(\Gamma)$ is the groupoid of co-cliques of the Coxeter graph $\Gamma$. Further, he proved that if $W$ does not contain $\U_3$ as a parabolic subgroup, then $\Spe(W)=\Inn(W)$. For a finite rank irreducible RACG $W$, M\"uhlherr \cite{Mu-98} explicitly determined $\Spe(W)$, and proved that $\Spe(W) = \Inn(W)$ if and only if the neighbourhood of each vertex in $\Gamma$ is an independent subset. For $\U_n$, he \cite{Mu-96} proved that $\Aut(\U_n)=\Spe(\U_n)\rtimes \s_n$. Bahls \cite{Ba-06} proved that for an even, large-type Coxeter group $W$, whose finite diagram $\mathcal{V}$ has no vertex $v$ such that $\mathcal{V}\setminus \{v\}$ consists of more than two connected components, $\Aut(W)=\Spe(W)\rtimes \diag(W)$. Ryan \cite{Ry-07} obtained that for rigid Coxeter groups whose labels are all odd, $\Aut(W)=\Inn(W)\rtimes \diag(W)$. Caprace and Minasyan \cite{CM-13} studied parabolic-preserving automorphisms, and established criteria for determining when such automorphisms are inner-by-graph.

A comprehensive literature survey on the automorphisms and the isomorphism problem of the Coxeter groups is presented in \cite{Ba-05, Mu-06, Nu-14}. A few conjectures regarding the solutions of the reduced isomorphism problem (referred to as the Twist-Conjecture) are proposed in \cite{Mu-06}. Recent progress, along with counterexamples to the Twist Conjecture, can be found in \cite{CP-10, CP-11, HP-18, MM-08, Pr-21, RT-08, We-11}. Some quotient isomorphism invariants for finitely generated Coxeter groups have been investigated in \cite{MRT-08}, while a modern geometric approach to the isomorphism problem has recently been introduced in \cite{SRS-24}.

The primary aim of this paper is to contribute to the existing literature on the automorphism of Coxeter groups. Let $\sw$ denote the family of Coxeter groups $W$ admitting a finite diagram that is a star graph (the complete bipartite graph $K_{1, n}$). We restrict our focus to groups of rank $\ge 3$. The automorphisms of rank 1 (i.e., $\mathbb{Z}_2$) and rank 2 (i.e., dihedral groups) Coxeter groups are elementary, and their structures are well-established. Although the automorphisms of rank 3 Coxeter groups were described by Franzsen \cite{Fr-01, Fr-02}, we include the rank 3 case here to provide a unified treatment of the star-shaped diagram. While a subfamily of $\sw$ consisting of odd Coxeter groups was previously investigated in \cite{NS-21}, our approach provides a uniform characterization for the automorphism groups of all members within this broader class, including those with even labels. 

We begin by establishing the structural foundations for star-shaped Coxeter groups in Section 2, where we recall essential results on parabolic subgroups and the centralizers of generators. We also define three specific types of automorphisms: transvections, and two type partial conjugations. In Section 3, we focus on the interactions between transvections and partial conjugations, as well as the subgroups they generate. In Section 4, we present the explicit structural description of $\Aut (W)$ (Theorem $\ref{Automorphism group of star graph}$), and derive splitting criteria for the associated short exact sequence $1 \to \Inn (W) \to \Aut (W) \to \Out (W) \to 1$ (Theorem $\ref{Splitting-of-AutW}$). We conclude Section 4 by verifying the $R_\infty$-property (Theorem \ref{R-infinity Property}) for these groups using Moussong's criteria for hyperbolicity. Finally, Section 5 addresses the rigidity (Theorem $\ref{rigidity theorem for W(T)}$) of these systems, providing a comprehensive solution to the isomorphism problem for star-shaped Coxeter groups.

Throughout this paper, we adopt the following notational conventions. For a group $G$, $\Aut(G)$, $\Inn(G)$, and $\Out(G)$ denote the groups of all automorphisms, inner automorphisms, and outer automorphisms, respectively. The center of $G$ is denoted by $\Z(G)$. For any $g \in G$, $\widehat{g}$ represents the inner automorphism induced by $g$, defined by $\widehat{g} (x) = x^g = gxg^{-1}$, for all $x\in G$. For a subgroup $H$ of $G$, the centralizer of $g$ in $H$ is written as $\C_H(g)$. For a subset $X \subset H$, the normal closure of $X$ in $H$ is denoted by $\langle\langle X \rangle\rangle_H$. For some positive integer $n$, we denote by $G^n$ the direct product of $n$ copies of $G$. We denote the cyclic group of order $n$ by $\mathbb{Z}_n$, the multiplicative group of relatively prime integers modulo $n$ by $\mathbb{Z}_n^\times$, the symmetric group on $n$ symbols by $S_n$, and the dihedral group of order $2m$ by $D_{2m}$.

\section{Structural Foundations}\label{Structural Foundations}
Let $\sw$ be the family of star-shaped Coxeter groups of finite rank $\geq 3$. For any $W \in \sw$, the star-shaped structure of the associated finite diagram allows for a specific indexing of the generators that is central to our work.  We fix a generating set $S = \{s_0, s_1, \dots, s_n\}$, where $s_0$ denotes the central vertex and $s_1,s_2,
\ldots, s_n$ denote the leaf vertices. For convenience, we define the index sets $I := \{1, 2, \dots, n\}$ and $\bar{I}=:I\cup \{0\}$. Under this convention, $W$ admits the following presentation:
\begin{equation}\label{presentation of star graph}
    W := \langle S  \mid s_i^2=1=(s_0s_{j})^{m(s_0,s_j)}, i\in \bar{I}, j\in I \rangle.
\end{equation}
For simplicity, we write $m_j$ to denote the label $m(s_0,s_j)$. We consistently adopt the presentation $\eqref{presentation of star graph}$ for any Coxeter group $W \in \sw$ throughout the article.  To facilitate our computations, we partition the index set $I$ based on the parity of the edge labels $m_i$. 
$$I_{\odd} := \{i \in I \mid m_i~\text{is odd}\}, \quad \text{and} \quad I_{\even} := \{j \in I \mid m_j~\text{is even}\}.$$
We denote the distinct odd labels by $m_{o,1}, \dots, m_{o,h_o}$ with multiplicities $\alpha_1, \dots, \alpha_{h_o}$, and the distinct even labels by $m_{e,1}, \dots, m_{e,h_e}$ with multiplicities $\beta_1, \dots, \beta_{h_e}$.  Defining $n_o := |I_{\odd}|$ and $n_e := |I_{\even}|$, we have the following relations:
\begin{equation*}
\sum_{k=1}^{h_o} \alpha_k = n_o, \quad \sum_{l=1}^{h_e} \beta_l = n_e, \quad \text{and} \quad n_o + n_e = n.
\end{equation*}

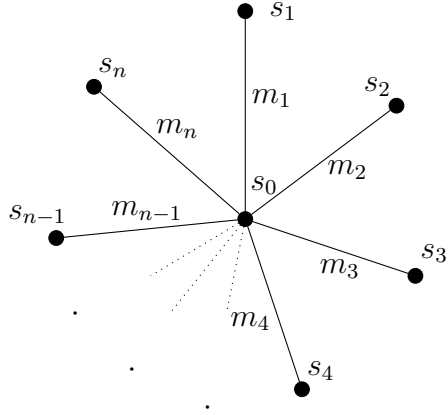
\begin{center}
\begin{figure}[ht]
    \centering      
   \begin{tikzpicture}
	\begin{pgfonlayer}{nodelayer}
		\node [style=new style 0, scale=0.5, color=black] (0) at (0, 0) {};
		\node [style=new style 0, scale=0.5, color=black] (1) at (0, 2.75) {};
		\node [style=new style 0, scale=0.5, color=black] (2) at (2, 1.5) {};
		\node [style=new style 0, scale=0.5, color=black] (3) at (2.25, -0.75) {};
		\node [style=new style 0, scale=0.5, color=black] (4) at (0.75, -2.25) {};
		\node [style=new style 0, scale=0.5, color=black] (6) at (-2.5, -0.25) {};
		\node [style=new style 0, scale=0.5, color=black] (7) at (-2, 1.75) {};
		\node [style=none] (8) at (-0.5, -2.5) {$\cdot$};
		\node [style=none] (9) at (-1.5, -2) {$\cdot$};
		\node [style=none] (10) at (-2.25, -1.25) {$\cdot$};
		\node [style=none] (11) at (0.5, 2.75) {$s_1$};
		\node [style=none] (12) at (1.75, 1.75) {$s_2$};
		\node [style=none] (13) at (2.5, -0.45) {$s_3$};
		\node [style=none] (14) at (1, -2) {$s_4$};
		\node [style=none] (16) at (-2.75, 0.05) {$s_{n-1}$};
		\node [style=none] (17) at (-1.75, 2) {$s_{n}$};
		\node [style=none] (18) at (0.35, 1.6) {$m_1$};
		\node [style=none] (19) at (1.35, 0.65) {$m_2$};
		\node [style=none] (20) at (1.25, -0.65) {$m_3$};
		\node [style=none] (21) at (0.07, -1.35) {$m_4$};
		\node [style=none] (23) at (-0.9, 1.2) {$m_{n}$};
		\node [style=none] (24) at (-1.3, 0.1) {$m_{n-1}$};
		\node [style=none] (26) at (0.25, 0.45) {$s_0$};
		\node [style=none] (27) at (-1.25, -0.75) {};
		\node [style=none] (28) at (-1, -1.25) {};
		\node [style=none] (29) at (-0.25, -1.25) {};
	\end{pgfonlayer}
	\begin{pgfonlayer}{edgelayer}
		\draw (0) to (1);
		\draw (0) to (2);
		\draw (0) to (3);
		\draw (0) to (4);
		\draw (0) to (6);
		\draw (0) to (7);
		\draw [dotted] (0) to (27.center);
		\draw [dotted] (0) to (28.center);
		\draw [dotted] (0) to (29.center);
	\end{pgfonlayer}
\end{tikzpicture}
 \end{figure}
 \end{center}

Since the automorphisms of odd Coxeter groups (those with only odd edge labels, i.e., $n_e = 0$) were investigated in \cite{NS-21}, we henceforth assume the existence of at least one even edge label ($n_e \geq 1$). 

Diagram automorphisms of $W$ are precisely the permutations of the leaf vertices that preserve the labels of the edges incident to $s_0$. For convenience, we identify each graph automorphism $\rho$ of $\mathcal{V}_{(W,S)}$ with its induced group automorphism in $\Aut(W)$. Let $\diag_o$ and $\diag_e$ denote the subgroups of diagram automorphisms that act on the leaf vertices with odd-labeled and even-labeled edges, respectively, while fixing the remaining leaf vertices. The group $\diag(W)$ of diagram automorphisms of $W$ then admits the following direct product decomposition.

\begin{lemma}\label{structure of diagonal automorphisms}
    $\diag(W) =\diag_o\times\diag_e \cong \prod_{k=1}^{h_o} \s_{\alpha_k} \times \prod_{l=1}^{h_e} \s_{\beta_l}.$
\end{lemma}

The center of any infinite irreducible Coxeter group is trivial \cite[Lemma 2.16]{Fr-01}. This immediately yields the following result:

\begin{lemma}\label{W has trivial center} 
For any $W\in \sw$, the center $\Z(W)$ is trivial.  Consequently, $\Inn(W)\cong W$.
\end{lemma}

We define the index set $I_2:= \{i \in I \mid m_i = 2\} \subseteq  I_{\even}$, with cardinality $n_2$, to identify the leaf generators that commute with the central vertex. This categorization is essential for analyzing the parabolic subgroups of $W$. To this end, we recall the following results from the literature:

\begin{lemma}\label{finite subgroup is contained in finite parabolic} \cite[Ex.2d, p.130]{Bo-02} Finite subgroups of an infinite Coxeter group are contained in some finite parabolic subgroup. 
\end{lemma}

\begin{lemma}\label{maximals are not conjugate}
\cite[Lemma 11 and Corollary 12]{FH-03} Maximal finite standard parabolic subgroups are not conjugate to subgroups of any other finite standard parabolic subgroup; consequently, they are also maximal finite subgroups.
\end{lemma}

Together, these results characterize the maximal finite subgroups of $W$ as follows:

\begin{lemma}\label{maximal parabolic subgroups are not conjugate}
\begin{enumerate}
\item[(i)] The maximal finite standard parabolic subgroups of $W$ are $W_i:=\langle s_0,s_i\rangle$, for all $i\in I$, where $$W_i \cong \begin{cases} 
\mathbb{Z}_2 \times \mathbb{Z}_2, &  \text{ if } i\in I_2,\\
D_{2m_i},& \: \text{if } i\in I\setminus I_2.
\end{cases}$$

\item[(ii)] Whenever $i\neq j$, the parabolic subgroups $W_i$ and $W_j$ are not conjugate.
 
\item[(iii)] Each $W_i$ contains a unique maximal element $\Delta_i$, defined by:
$$\Delta_i:=\begin{cases}
s_i(s_0s_i)^{\frac{m_i-1}{2}},& \text{if $i\in I_\odd$},\\
(s_0s_i)^{\frac{m_i}{2}},& \text{if $i\in I_\even$}.
\end{cases}$$
Furthermore, $\Delta_i$ lies in the center $\Z(W_i)$ if and only if $i \in I_{\even}$.
\end{enumerate}
\end{lemma}

The centralizers of generators in any Coxeter groups were characterized by  Brink \cite[Theorem, p.466]{Br-96} and Bahls \cite[Theorem 2.5]{Ba-05} in terms of root systems, and specific paths and loops in the associated Coxeter diagram, respectively. Applying these characterization to star-shaped Coxeter group $W\in \sw$, we obtain the following explicit structures:

\begin{lemma}\label{structure of the centralizer for s0}
Let $W \in \sw$. Then, for each $s_i\in S$, the centralizer $\C_W(s_i)$ is a RACG. In particular,
$$\C_W(s_i)=
\begin{cases}
\langle s_i \rangle  \cong \mathbb{Z}_2, & \text{if}\; i\in I_\odd,\\
\langle s_i \rangle \times \langle \Delta_i \rangle  \cong \mathbb{Z}_2\times \mathbb{Z}_2, &\text{if}\; i\in I_\even,\\
\langle  s_0 \rangle \times \langle \Delta_j\mid j\in I_\even \rangle  \cong \mathbb{Z}_2\times \U_{n_e}, &\text{if}\; i=0.\\
\end{cases}$$
\end{lemma}

We denote the subgroup $\langle \Delta_j\mid j\in I_\even\rangle$ of $\C_W(s_0)$ by $\Delta$, which is isomorphic to $\U_{n_e}$. A few observations can be readily verified: 
\begin{equation*}
     \C_W(s_0)\cap W_i=\begin{cases}
        \langle s_0\rangle, & \text{if $i\in I_\odd$},\\
        \langle s_0\rangle \times \langle \Delta_i\rangle, & \text{if $i\in I_\even$},
\end{cases} \quad \text{and} \quad \Delta\cap W_i=\begin{cases}
        1, & \text{if $i\in I_\odd$},\\
        \langle \Delta_i\rangle, & \text{if $i \in I_\even$}, 
\end{cases}
\end{equation*} 
i.e., for each $i\in I$, $\Delta\cap W_i\leq \C_W(s_0)\cap\C_W(s_i)$.

In any Coxeter system, two generators are conjugate if and only if they are connected by a path of odd-labeled edges \cite[Lemma 3.6]{BMMN-02}. For our star-shaped Coxeter groups $W\in \sw$, the central vertex $s_0$ is connected to a leaf $s_i$ by an odd-labeled edge if and only if $i\in I_\odd$, and no edges exist between distinct leaf vertices. This observation leads to the following result:

\begin{lemma}\label{conjugacy lemma through odd graphs}
Let $W\in \sw$ be a star-shaped Coxeter group with the generating set $S$. Then $s_0$ is conjugate to $s_i$ if and only if $i\in I_\odd$, whereas for each $j \in I_\even$, $s_j$ is not conjugate to any other generator in $S$. Specifically, the partition of $S$ into conjugacy classes consists of the set  $\{s_i \mid i \in \bar{I}\setminus I_\even\}$, and the singletons $\{s_j\}$ for each $j \in I_\even$.
\end{lemma}

\begin{corollary}\label{quotient by generators}
Since conjugate elements have the same normal closures, the structures of the quotients $W / \langle \langle s \rangle \rangle_W$ for $s \in S$ are given as follows:
\begin{itemize}
 \item[(i).]  If $i \in \bar{I}\setminus I_\even$, then $W / \langle \langle s_i \rangle \rangle_W  \cong W_{I_\even} \cong \U_{n_e}$.
  \item[(ii).] If $j \in I_\even$, then $W / \langle \langle s_j \rangle \rangle_W$ is isomorphic to the parabolic subgroup $W_{S \setminus \{s_j\}}$.
\end{itemize}
\end{corollary}

\begin{lemma}\label{intersection of maximal finite subgroup is again parabolic}\cite[Corollary 9]{FH-03} 
A subgroup of an infinite Coxeter group that can be realized as the intersection of a collection of maximal finite subgroups is a parabolic subgroup.
\end{lemma}

As an application of these results, we characterize the image of $s_0$ under $\Aut(W)$.
 
\begin{lemma}\label{s0 maps to s0 under automorphism of W}
Let $\varphi\in \Aut(W)$. Then $\varphi(s_0) = s_0$  up to inner automorphism.
\end{lemma}

\begin{proof}
Note that $\varphi \big ( \langle s_0 \rangle \big ) = \varphi \big (  \bigcap_{ i = 1}^n W_i \big )   =  \bigcap_{i=1}^n \varphi  ( W_i) $. Since each $W_i$ is a maximal finite subgroup, so is their images $\varphi(W_i)$. By Lemma \ref{intersection of maximal finite subgroup is again parabolic}, $\varphi(\langle s_0 \rangle)$ is a parabolic subgroup of rank one, i.e., $\varphi(s_0)$ is conjugate to some $s_i \in S$. 
By Corollary \ref{quotient by generators}, the quotient $W / \langle \langle s_0 \rangle \rangle_W$ is a universal Coxeter group, while $W / \langle \langle s_j \rangle \rangle_W$ is not for $j \in I_\even$ (it contains at least one edge with a finite label). Since $\varphi$ induces an isomorphism of these quotients, $\varphi(s_0)$ cannot be conjugate to any $s_j \in I_\even$. Consequently, $\varphi(s_0)$ must be conjugate to $s_0$, which implies $\varphi(s_0) = s_0$ up to an inner automorphism.
\end{proof}

\begin{lemma}\label{maximal parabolic subgroups goes to its conjugate}
Let $\varphi\in \Aut(W)$. Then, for each $i \in I$, $\varphi(W_i)$ is conjugate to $W_i$ up to diagram automorphisms.
\end{lemma}

\begin{proof}
By Lemma \ref{maximals are not conjugate} and Lemma \ref{intersection of maximal finite subgroup is again parabolic}, $\varphi(W_i)$ is a maximal finite parabolic subgroup. Thus $\varphi(W_i)$ is conjugate to $ W_j$ for some $j \in I$. Since $W_i \cong W_j$, their corresponding edge labels must be equal, i.e., $m_i = m_j$. This equality ensures the existence of a diagram automorphism which swaps the subgroups $W_i$ and $W_j$. Consequently, we conclude that, up to diagram automorphisms, $\varphi(W_i)$ is conjugate to $W_i$.
\end{proof}

 
\begin{corollary}\label{reflections map to its own conjugate or to maximal element}
Let $\varphi\in\Aut(W)$. Then, up to diagram automorphisms, the following hold:
    \begin{enumerate}
        \item[(i)] If $i\in I\setminus I_2$, $\varphi(s_i)$ is conjugate to $s_i$. 
        \item[(ii)] If $i\in I_2$, $\varphi(s_i)$ is conjugate to $s_i$ or $s_0s_i$.
\end{enumerate} 
\end{corollary}

The preceding results establish that every automorphism of $W$ stabilizes the collection of maximal finite parabolic subgroups, up to inner and diagram automorphisms. Consequently, to characterize $\Aut(W)$, we now define the automorphisms that preserve these subgroups: transvections and partial conjugations.
\medskip

\textbf{Transvections.} For each $i\in I_2$, we define a map $\tau_i: S \to W$, given by 
$$\tau_i(s_j):=\begin{cases}
s_0s_i,&\text{if } j=i,\\
s_j,&\text{if } j\neq i.
\end{cases}$$
This map extends to an involutive automorphism of $W$, which we refer to as a \textit{transvection}. Let $\T := \langle \tau_i \mid i \in I_2 \rangle$ be the subgroup of $\Aut (W)$ generated by all such transvections. The following properties of $\T$ are immediate from the definitions.

\begin{lemma}\label{Structure of T}
Let $n_2 = |I_2|$. Then the following assertions hold:
\begin{itemize}
\item[(i).] Any two transvections $\tau_i$ and $\tau_j$ commute. Consequently, $\T \cong (\mathbb{Z}_2)^{n_2}$.
\item[(ii).] For each $\tau \in \T$, the centralizer of $s_0$ is invariant; i.e., $\tau(\C_W(s_0)) = \C_W(s_0)$.
\item[(iii).] Any $\tau \in \T$ fixes $\C_W(s_0)$ pointwise if and only if $\tau = \Id$.
\end{itemize}
\end{lemma}

\textbf{Partial Conjugation of Type $1$.} For each $i \in I \setminus I_2$ and each integer $1 \leq t < m_i$ with $\gcd(t, m_i) = 1$, we define a map $\phi_{i,t}: S \to W$ given by: 
$$\phi_{i,t}(s_j) := \begin{cases} 
 s_i (s_0 s_i)^{t-1}, & \text{if } j = i,\\
s_j, & \text{if } j \neq i.
 \end{cases}$$
These maps extend to automorphisms of $W$, which we refer to as \textit{partial conjugations of type $1$}. They belong to the broader concept of angle deformations \cite{MM-08}. 

Notice that $\phi_{i,t}$ maps $s_i$ to one of its conjugates within the parabolic subgroup $W_i$ while fixing all other generators $s_j$ for all $j \neq i$. For each $i \in I \setminus I_2$, let $\p_{1,i} := \{ \phi_{i,t} \mid 1 \leq t < m_i,~ \gcd(t, m_i) = 1 \}$. We define $\p_{1}:=\langle \p_{1,i}\mid i\in I\setminus I_2\rangle$ to be the subgroup of $\Aut(W)$ generated by all partial conjugations of type $1$. The following structural properties of $\p_1$ are immediate.

\begin{lemma}\label{Structure of P}
The following assertions hold:
\begin{itemize}
\item[(i).]  For each $i \in I \setminus I_2$, $\p_{1,i} \cong \mathbb{Z}_{m_i}^\times$. In particular, $\p_1 \cong \prod_{k=1}^{h_o} \big(\mathbb{Z}_{m_{o,k}}^\times\big)^{\alpha_k} \times \prod_{l=1}^{h_e}\big( \mathbb{Z}_{m_{e,l}}^\times\big)^{\beta_l}.$
\item[(ii).]  $\p_1$ fixes the centralizer $\C_W(s_0)$ of $s_0$ pointwise.
\end{itemize}
\end{lemma}

\textbf{Partial Conjugations of Type $2$.} The final type of automorphisms consists of partial conjugations determined by the even-labelled edges. For each $i \in I_\even$ and $j \neq i \in I$, we define the automorphism $\psi_{i,j}$ by its action on the generating set $S$ given by:
$$\psi_{i,j}(s_k) := 
\begin{cases} 
\Delta_i s_j \Delta_i, & \text{if } k = j,\\
s_k, & \text{if } k \neq j.
\end{cases}$$
Additionally, if $i \in I_2$, we define a second class of partial conjugation $\sigma_{i,j}$ (for all $j \neq i \in I$) given by: 
$$\sigma_{i,j}(s_k) := 
\begin{cases} 
s_i s_j s_i, & \text{if } k = j,\\
s_k, & \text{if } k \neq j.
\end{cases}$$
These automorphisms, which we refer to as {\it partial conjugations of type} $2$, are related to the diagram twisting introduced in \cite{Mu-06} and the special automorphisms of irreducible right-angled Coxeter groups studied in \cite{Mu-98}. 

For each $i \in I_\even$, we define the local subgroups $\p_{2,i} := \langle \psi_{i,j} \mid j \neq i \rangle$ if $i \in I_{\text{even}} \setminus I_2$, and $\p_{2,i} := \langle \psi_{i,j}, \sigma_{i,j} \mid j \neq i \rangle$ if $i \in I_2$. In both cases, $\p_{2,i}$ is an elementary abelian $2$-group satisfying:
$$\p_{2,i}\cong 
\begin{cases}
(\mathbb{Z}_2)^{2n-n_2-1},& \text{if}\; i\in I_2,\\
(\mathbb{Z}_2)^{n-1},& \text{if}\; i\in I_\even \setminus I_2.
\end{cases}$$
We define $\p_2:= \langle \p_{2, i} \mid i \in I_\even \rangle$ to be  the subgroup of $\Aut(W)$ generated by all partial conjugations of type 2. Its full structural characterization is deferred to the next section. We conclude this section with the following easy observations:
\begin{itemize}
\item $\p_2$ keeps $\Delta$ invariant, and conjugate the generators by elements of $\C_W(s_0)$.
\item $\sigma_{i,j} = \psi_{i,j}$ if and only if $j \in I_2$.
\item For each $j\in I\setminus I_2$, $\psi_{i,j}\sigma_{i,j}=\phi_{j,{m_j-1}} \in \p_1$.
\end{itemize}

\section{Characterizing Subgroups of Aut(W)}\label{Important subgroups of Aut(W)}
Building upon the preliminaries in Section 2, we now characterize the specific subgroups of $\Aut(W)$ generated by transvections and partial conjugations. We establish the commutation relations between these subgroups and determine their precise structure. 

\begin{lemma}\label{T and P commute}
The subgroup $\langle \T, \p_1 \rangle \le \Aut(W)$ is the internal direct product $\T \times \p_1$.
\end{lemma}

\begin{proof}
Let $\tau_i \in \T$ and $\phi_{j,t} \in \p_1$ for some $i\in I_2$, $j\in I\setminus I_2$. By definition, $\tau_i$ acts non-trivially only on the parabolic subgroup $W_i$, while $\phi_{j,t}$ acts non-trivially only on $W_j$. Also, both $\tau_i$ and $\phi_{j,t}$ fix the intersection $W_i \cap W_j = \langle s_0\rangle$. Since $I_2 \cap (I \setminus I_2) = \emptyset$, these automorphisms have disjoint support and thus commute. The triviality of the intersection follows immediately, establishing the direct product.
\end{proof}

\begin{lemma}\label{local adjustment}
Let $\varphi \in \Aut(W)$ be an automorphism of $W$ such that $\varphi(W_i) = W_i$ for all $i \in I$. Then for any $j \in I$, there exists an automorphism $\eta_j \in \T \times \p_1$ such that
$$\eta_j \circ \varphi|_{W_j} = \Id_{W_j} \quad \text{and} \quad \eta_j \circ \varphi|_{W_k} = \varphi|_{W_k} \text{ for all } k \neq j.$$
\end{lemma}

\begin{proof}
Suppose that $\varphi(W_i) = W_i$ for all $i \in I$. Since $\varphi$ must fix the intersection of all $W_i$, we have $\varphi(s_0)=s_0$. If $j \in I_2$, we must have either $\varphi(s_j) = s_j$ or $\varphi(s_j) = s_0s_j$. If $\varphi(s_j) = s_j$, we set $\eta_j = \Id$; if $\varphi(s_j) = s_0s_j$, we set $\eta_j = \tau_j$. In either case, $(\eta_j \circ \varphi)|_{W_j} = \Id_{W_j}$. On the other hand, if $j \in I \setminus I_2$, we get $\varphi(s_j) = s_j(s_0s_j)^{t-1}$ for some $1 \leq t < m_j$ with $\gcd(t, m_j) = 1$. By setting $\eta_j = (\phi_{j,t})^{-1} \in \p_{1,j}$, we obtain $(\eta_j \circ \varphi)|_{W_j} = \Id_{W_j}$. Finally, since the support of any $\eta_j$ is contained in $W_j$, it fixes the generators $s_t$ for all $t \neq j$. Therefore, $\eta_j \circ \varphi|_{W_t} = \varphi|_{W_t}$ for all $t \neq j$.
\end{proof}
The following corollary is now immediate.
\begin{corollary}\label{automorphisms that preserve the maximal standard parabolic subgroups}
$\varphi \in \T\times \p_1$ if and only if $\varphi (W_i)=W_i$ for all $ i\in I$.
\end{corollary}

To determine the complete structure of $\p_2$, we consider below some subgroups of it.
\begin{itemize}
\item Let $\p_{2}^e:=\langle \psi_{i,j}\mid i,j\in I_\even,~i\neq j\rangle$ if $n_e \geq 2$, and set $\p_{2}^e = 1$ otherwise.
\item Let $\p_2^o:=\langle \psi_{i,j}\mid i\in I_\even,~j\in I_\odd\rangle$ if $n_o\geq 1$, and set $\p_2^o=1$ otherwise.
\end{itemize}
By construction, the subgroup $\p_2^e$ preserves $\Delta$ set-wise, whereas $\p_2^o$ fixes it pointwise.

For an irreducible RACG group $G$, a presentation for the special automorphism group $\Spe(G)$ is known \cite{Mu-96, Mu-98}.  By applying these results to $\Delta  \cong \U_{n_e}$, we obtain just a generating set of $\Spe(\Delta)$ in the following lemma, omitting the relations, which are typographically lengthy and not required for our purposes. This explicit generating set subsequently allows us to determine the structure of $\p_2^e$.

\begin{lemma}
    The group of special automorphisms $\Spe(\Delta)$ is generated by the set $\{\theta_{i,j}\mid i,j\in I_\even,~i\neq j\}$, where the action of each $\theta_{i,j}$ on the generator of $\Delta$ is given by
\begin{equation*}
\theta_{i,j}(\Delta_k):=
\begin{cases}
\Delta_i\Delta_j\Delta_i,&\text{if } k=j,\\
\Delta_k,&\text{if } k\neq j.
\end{cases}
\end{equation*}
\end{lemma}

%

\begin{lemma}\label{structure of E}
$\p_2^e \cong \Spe(\U_{n_e})$.
\end{lemma}

\begin{proof}
We begin by noting that for each pair of distinct indices $i, j \in I_\even$, the restriction $\psi_{i,j}|_{\Delta}$ coincides with the generator $\theta_{i,j}$ of $\Spe(\Delta)$. This correspondence induces a canonical surjective homomorphism $\res :\p_2^e\to \Spe(\Delta)$ given by $$\res(\psi_{i,j})=\theta_{i,j}.$$ To establish injectivity, let $\psi\in \ker(\res)$. By the definition of $\p_2^e$, we have $\psi(s_i) = x_i s_i x_i^{-1}$ for some $x_i \in \Delta$ for each $i \in I_\even$. Since $\psi(s_0) = s_0$, we obtain that $$\psi(\Delta_i)=(s_0x_is_ix_i^{-1})^\frac{m_i}{2} = x_i\Delta_i x_i^{-1}.$$ On the other hand, since $\psi$ is in the kernel, it must fix $\Delta$ pointwise; i.e., $\psi(\Delta_i) = \Delta_i$ for all $i$. Hence, $x_i\in \Delta\cap \C_W(\Delta_i) =\langle \Delta_i\rangle\leq \C_W(s_i)$. It follows that $\psi(s_i) = x_i s_i x_i^{-1} = s_i$ for all $i \in I_\even$. Since $\psi$ also fixes the generators $s_0$ and $s_j$ for all $j \in I_\odd$, we conclude that $\psi = \Id$. Thus, $\ker(\res) = 1$, establishing the isomorphism $\p_2^e\cong \Spe(\U_{n_e})$
\end{proof}

\begin{corollary}\label{characterization of E}
$\psi\in \p_2^e$ if and only if
$
        \psi(s_i)=\begin{cases}
        s_i, & \text{if $i\in \bar{I}\setminus I_\even$,}\\
            x_is_ix_i^{-1}, & \text{if $i\in I_\even$,}
        \end{cases}
$
where $x_i\in \Delta$. 
\end{corollary}

\begin{proof}
We have $\psi(\Delta_i)=x_i\Delta_ix_i^{-1}$ for each $i\in I_\even$. Therefore, the restriction of $\psi$ on $\Delta$ induces an automorphism of $\Delta$. Then, by Lemma $\ref{structure of E}$, we get $\psi|_{\Delta}\in \Spe(\U_{n_e})$, and consequently $\psi\in \p_2^e$. The converse follows from the definition of $\p_2^e$.
\end{proof}

\begin{lemma}
$\p_2^o\cong (\U_{n_e})^{n_o}$.
\end{lemma}

\begin{proof}
For each $j \in I_\odd$, let $\p_2^o(j) := \langle \psi_{i,j} \mid i \in I_\even \rangle \leq \p_2^o$. Define a map $f_j: \Delta \to \p_2^o (j)$ by $f_j(\Delta_i) = \psi_{i,j}$ for each $i \in I_\even$. Since $\Delta$ is a universal Coxeter group, this map extends to a surjective group homomorphism. To show injectivity, let $d =\Delta_{i_1}\Delta_{i_2}\cdots\Delta_{i_k}\in \ker(f_j)$. Then $\psi_{i_1,j}\psi_{i_2,j}\cdots\psi_{i_k,j}=\Id$, and by definition, $d\in \C_W(s_j)\cap \Delta=1$. Hence, $\ker(f_j)=1$. Thus, $\p_2^o (j) \cong \Delta \cong \U_{n_e}$ for every $j \in I_{\text{odd}}$.

Next, we consider the structure of $\p_2^o = \langle \p_2^o(j) \mid j \in I_\odd \rangle$. For distinct indices $j, l \in I_\odd$, the generators of $\p_2^o(j)$ and $\p_2^o(l)$ commute because their actions are localized to disjoint sets of leaf generators $\{s_j\}$ and $\{s_l\}$ respectively. Furthermore, any element of $\langle \p_2^o(l) \mid l \in I_\odd \setminus \{j\} \rangle$ fixes $s_j$, while any non-trivial element of $\p_2^o(j)$ acts non-trivially on $s_j$. Consequently, the intersection $\p_2^o(j) \cap \langle \p_2^o(l) \mid l \in I_\odd \setminus \{j\} \rangle$ is trivial. This ensures that $\p_2^o$ is the direct product of the subgroups $\{\p_2^o(j)\}_{j \in I_\odd}$. We therefore conclude that $\p_2^o = \prod_{j \in I_\odd} \p_2^o(j) \cong (\U_{n_e})^{n_o}.$
\end{proof}

\begin{lemma} \label{intersection of P and C}
The intersection of $\p_1$ and $\p_2$ is given by
\begin{equation*}
\p_3:=\p_1\cap \p_2=\begin{cases}
\langle \phi_{j,{m_j-1}}\mid j\in I\setminus I_2\rangle, & \text{if $n_2\geq 1$},\\
1, & \text{if $n_2=0$}.\\
\end{cases}
\end{equation*}
\end{lemma}

\begin{proof}
 Let $\psi\in \p_3$. Then $\psi(s_0)=s_0$ and $\psi(s_i)=x_is_ix_i^{-1}$, where $x_i\in W_i$ for each $i\in I$. If $n_2\geq 1$, then $x_i\in \C_W(s_0)\cap W_i$ for all $i\in I$. This intersection is precisely $\langle s_0 \rangle$ for $i \in I_\odd$ and $\langle s_0, \Delta_i \rangle$ for $i \in I_\even$. Furthermore, $\langle s_0, \Delta_i\rangle\leq \langle s_0\rangle\times \C_W(s_i)$ for $i\in I_\even$. Hence, without loss of generality, we may assume that $x_i\in \langle s_0\rangle$ for each $i\in I$. Now, if $x_i = s_0$, then $\psi(s_i) = s_0 s_i s_0 = s_i (s_0 s_i)^{m_i-2} =  \phi_{i,{m_i-1}}(s_i)$. Therefore, $\psi\in \langle \phi_{i,{m_i-1}}\mid i\in I\setminus I_2\rangle$. In case $n_2=0$,  we have $x_i\in \Delta\cap W_i\leq \C_W(s_i)$ for all $i\in I$.  Consequently, $\psi= \Id$. Since $\psi_{i,j}\sigma_{i,j}=\phi_{j,{m_j-1}}$, for any $i\in I_2$ and $j\in I\setminus I_2$, the converse is immediate.
\end{proof}

\begin{lemma}\label{center of C is same as A}
$
    \Z(\p_2)=\begin{cases}
       \p_2, & \text{if}\; n_e=1,\\
    \p_3, & \text{if}\; n_e\geq 2.    
    \end{cases}
$
\end{lemma}

\begin{proof}
If $n_e=1$, the definition implies $\p_2$ is abelian, so $\Z(\p_2)=\p_2$. We may therefore assume $n_e \geq 2$.

Let $\psi\in \Z(\p_2)$. Then for each $i\in \bar{I}$, $\psi(s_i)=w_is_iw_i^{-1}$ where $w_i\in \C_W(s_0)$. Recall that (Lemma \ref{structure of the centralizer for s0}) $\C_W(s_0) = \langle s_0 \rangle \times \Delta$, where $\Delta = \langle \Delta_i \mid i \in I_\even \rangle$ is a universal Coxeter group. Each element of $\Delta$ thus possesses a unique reduced expression in the letters $\Delta_i$. We claim that $w_i\in \langle s_0\rangle$ for all $i$. 

Suppose on the contrary that $w_i = s_0^{\epsilon} \Delta_{i_1} \Delta_{i_2} \cdots \Delta_{i_k}$ is the unique reduced expression with $\epsilon \in \{0,1\}$ and $k \geq 1$. Since $\Delta_i$ commutes with $s_i$, we can assume without loss of generality that $i_k \neq i$.
Since $\psi_{j,i_k} \in \p_2$ maps $\Delta_{i_k}$ to $\Delta_j \Delta_{i_k} \Delta_j$ while fixing all other $\Delta_l$ ($l \neq i_k$) and $s_0$, it can be readily verified that the unique reduced expression of $\psi_{j,i_k}(w_i)$ always terminates with the letter $\Delta_j$. Hence, $\psi_{j,i_k}(w_i) \neq w_i$. 

We show that this inequality directly contradicts the centrality of $\psi$ in $\p_2$. Since $\psi \in \Z(\p_2)$ and $\psi_{j,i_k}$ is an involutive automorphism, the relation $\psi_{j,i_k} \psi \psi_{j,i_k} = \psi$ holds. Evaluating this identity at $s_i$ yields: $$w_is_iw_i^{-1} = \psi (s_i) =  \psi_{j,i_k} \psi \psi_{j,i_k} (s_i) =  \psi_{j,i_k} \psi (s_i) = \psi_{j,i_k} (w_is_iw_i^{-1}) = \psi_{j,i_k}(w_i)s_i\psi_{j,i_k}(w_i)^{-1}.$$ Therefore $w_i^{-1}\psi_{j,i_k}(w_i)\in \C_W(s_i)$. Furthermore, because $\psi_{j,i_k}$ keeps $\C_W(s_0)$ invariant, it follows that $w_i^{-1}\psi_{j,i_k}(w_i)\in \C_W(s_i) \cap \C_W(s_0)$. This intersection which is $\{1\}$ if $i \in I_\odd$ and $\{1, \Delta_i\}$, if $i \in I_\even$. 

However, under the action of $\psi_{j,i_k}$, every letter $\Delta_i$ is mapped to an odd-length word (either of length $1$ or length $3$). Consequently, the total word length of $\psi_{j, i_k}(w_i)$ has the same parity as that of $w_i$, which implies that  $w_i^{-1}\psi_{j,i_k}(w_i)$ must expand to an even-length word in $\Delta$. Because $\Delta_i$ is a word of length $1$, it is excluded.
Thus, we conclude that $\psi_{j,i_k}(w_i)=w_i$,  which directly contradicts our earlier deduction that $\psi_{j,i_k}(w_i) \neq w_i$. Thus $w_i\in \langle s_0\rangle$ for each $i\in I$. Consequently, $\psi$ can be expressed in the form as $\prod_{i\in I}\big(\phi_{i,{m_i-1}}\big)^{\varepsilon_i}$ with $\varepsilon_i\in\{0,1\}$. Since $\phi_{i,{m_i-1}}\in \p_3$, we have  $\psi \in \p_3$, establishing $\Z(\p_2)\leq \p_3$.

Conversely, we show that $\p_3\leq \Z(\p_2)$. If $n_2=0$, then $\p_3=1\leq \Z(\p_2)$ is vacuously true. Assume $n_2\geq 1$. 
Let $\psi\in \p_2$. Then for each $i\in \bar{I}$, $\psi(s_i)=x_is_ix_i^{-1}$ where $x_i\in \C_W(s_0)$. Since $\phi_{i,{m_i-1}}$ acts as identity map on $\C_W(s_0)$ (Lemma \ref{Structure of P}), it follows that
\begin{align*}
\phi_{i,{m_i-1}}\psi \phi_{i,{m_i-1}}(s_i) & =\phi_{i,{m_i-1}}(s_0x_is_ix_i^{-1}s_0)=\phi_{i,{m_i-1}}(x_is_0s_is_0x_i^{-1})=x_is_ix_i^{-1}=\psi(s_i),&\\
\phi_{i,{m_i-1}}\psi \phi_{i,{m_i-1}}(s_j) & =\phi_{i,{m_i-1}}(x_js_jx_j^{-1})=x_js_jx_j^{-1}=\psi(s_j),\; \text{for all}\; j\neq i.&
\end{align*}
Because the identity $\phi_{i, m_i-1} \psi \phi_{i, m_i-1} = \psi$ holds on all generators, each $\phi_{i, m_i-1}$ commutes with $\psi$, implying $\phi_{i, m_i-1} \in \Z(\p_2)$. Since the group $\p_3$ is generated by these automorphisms, it follows that $\p_3 \leq \Z(\p_2)$. We conclude that $\Z(\p_2) = \p_3$.
\end{proof}

\begin{lemma}\label{structure of C}
$\p_2=(\p_2^o\rtimes \p_2^e) \times \p_3 \cong \begin{cases}
 (\U_{n_e})^{n_o}\rtimes \Spe(\U_{n_e}),&\text{if}\; n_2=0,\\
\big((\U_{n_e})^{n_o}\rtimes \Spe(\U_{n_e})\big) \times (\mathbb{Z}_2)^{n-n_2}, & \text{if}\; n_2\geq 1.      
\end{cases}$
\end{lemma}

\begin{proof}
Let $\psi\in \p_2$. Then, for each $i\in \bar{I}$, $\psi(s_i)=x_is_ix_i^{-1}$, where $x_i\in \C_W(s_0)$. Our goal is to bring these conjugating elements $x_i$ inside $\Delta$ modulo the action of $\p_3$.  If $n_2 = 0$, then $x_i$ already belongs to $\Delta$ for all $i$ and $\p_3 = 1$. 

Suppose instead that $n_2 \geq 1$. If $x_i \notin \Delta$ for some $i \in I$, then $x_i = x_i's_0 $, where $x_i' \in \Delta$. This gives us $\psi(s_i) = x_i' s_0 s_is_0 x_i'^{-1}$. Applying the involutive automorphism $\phi_{i, m_i-1}$, we obtain $$\phi_{i, m_i-1} \psi(s_i) = x_i' s_i x_i'^{-1}.$$ By applying these modifications only for those indices $i$ for which $x_i \notin \Delta$, we obtain an element $\gamma_1 \in \p_3$ (since each constituent $\phi_{i, m_i-1}$ belongs to $\p_3$) such that $$\gamma_1^{-1}\psi(s_i)=x_i's_ix_i'^{-1},$$ for all $i\in I$, where $x_i'\in \Delta$.

Note that a reduced expression of $x_i'$ does not involve the generators $s_j$ for any $j\in I_\odd$. Thus, by the definition of $\p_2^o$, we can find an element $\gamma_2\in \p_2^o$ such that 
 $$\gamma_2^{-1}\gamma_1^{-1}\psi(s_i)=\begin{cases}
 s_i, & \text{if $i\in I_\odd$},\\
        x_i's_ix_i'^{-1}, & \text{if $i\in I_\even$}.
\end{cases}$$
Note that the remaining partial conjugations are on the generators in $I_\even$. Therefore, by Corollary $\ref{characterization of E}$, $\gamma_2^{-1}\gamma_1^{-1}\psi=\gamma_3\in \p_2^e$, i.e., $\psi = \gamma_1\gamma_2\gamma_3\in \p_3\p_2^o\p_2^e$. This proves that $\p_2$ is generated by the subgroups $\p_2^o$, $\p_2^e$, and $\p_3$.

To establish the precise internal products, we analyze the mutual intersections and actions of these subgroups. By the respective definitions of $\p_2^o, \p_2^e$, and $\p_3$, it follows that $$\p_2^o \cap \p_2^e = 1 = \p_3\cap (\p_2^o\p_2^e).$$ Let $\psi_{i,j}$ be a generator of $\p_2^e$ and $\psi_{k,l}$ be a generator of $\p_2^o$, where $i,j,k\in I_\even$ (with $i\neq j$) and $l\in I_\odd$. It is easy to see that $$\psi_{i,j}\psi_{k,l}\psi_{i,j}=\begin{cases}
\psi_{k,l}, & \text{if $j\neq k$},\\
        \psi_{i,l}\psi_{k,l}\psi_{i,l}, & \text{if $j=k$}.
\end{cases}$$
It follows that $\p_2^e$ normalizes $\p_2^o$. Consequently, their product forms an internal semidirect product: $\p_2^o  \p_2^e = \p_2^o\rtimes \p_2^e$. Since $\p_3 \in Z(\p_2)$, we obtain the final decomposition $\p_2=\p_3\p_2^o\p_2^e=\p_3\times (\p_2^o\rtimes \p_2^e)$.
\end{proof}

As a consequence of Corollary $\ref{characterization of E}$ and Lemma $\ref{structure of C}$, we obtain the following result.

\begin{corollary}\label{lemma for the C-automorphisms}
An automorphism $\psi\in \p_2$ if and only if $\psi(s_i)=x_is_ix_i^{-1}$ for each $s_i \in S$, where $x_i\in \begin{cases}
        \C_W(s_0),&\text{if $n_2\geq 1$,}\\
        \Delta,&\text{if $n_2=0$}.
\end{cases}$
\end{corollary}

The following result is immediate by Lemma $\ref{intersection of P and C}$.

\begin{lemma}\label{structure of P and C}
    The group of partial conjugations $\p:=\langle \p_1,\p_2\rangle$ decomposes as the amalgamated direct product of $\p_1$ and $\p_2$ over their intersection $\p_3$, i.e., $\p=\p_1\times_{\p_3} \p_2$.
\end{lemma}

\begin{lemma}\label{structure of T P C}
The subgroup generated by transvections and partial conjugations is given by
$$\langle\T, \p \rangle = \p \rtimes\T.$$
\end{lemma}

\begin{proof}
 A direct computation shows that, for $1\leq i,j,k\leq n$ such that $i\in I_2$ and $j\in I_\even$ with $j\neq k$,$$n\tau_i\psi_{j,k}\tau_i=\begin{cases}
 \sigma_{j,k},& \text{if } j=i,\\
\psi_{j,k},&\text{if } j\neq i.
\end{cases}
$$ This proves that $\T$ normalizes $\p_2$. Also, $\p_2\cap \T=1$. Then, due to Lemma $\ref{T and P commute}$ and Lemma $\ref{structure of P and C}$, $\p\cap \T=1$ and, hence $$\langle \T,\p\rangle=\p \rtimes\T, \quad \text{when} \quad n_2\geq 1.$$ Next, if $n_2=0$, then $\T=1$ and the result is immediate. This completes the proof.
\end{proof}

\section{Structure of Aut(W)}\label{Structure of Aut(W) and Short Exact Sequences} 
In this section, we show that $\Aut(W)$ is generated by inner automorphisms, partial conjugations, transvections, and diagram automorphisms.

Since the conjugacy classes of the generators are not preserved under the action of $\T$, the following lemma is immediate.

\begin{lemma}
    $\Inn(W)\cap \T=1$.
\end{lemma}

\begin{lemma}\label{inner intersection P}
The intersection of inner automorphisms and partial conjugations is given by $\Inn(W)\cap\p =\big\langle \widehat{w}~\big|~w\in \C_W(s_0)\big\rangle$. In particular, $\p_1$ and $\p_2$ intersect $\Inn(W)$ as follows:
\begin{itemize}
    \item $\Inn(W)\cap \p_1= \langle\widehat{s_0}\rangle$.
    \item $\Inn(W)\cap \p_2=\begin{cases}
     \big\langle \widehat{w}~\big|~w\in \C_W(s_0)\big\rangle,&\text{if $n_2\geq 1$,}\\
        \big\langle \widehat{\Delta_j}\mid j\in I_\even\big\rangle,&\text{if $n_2=0$.}
    \end{cases}$
\end{itemize}  
\end{lemma}

\begin{proof}
Let $\psi\in \Inn(W)\cap \p $. Then $\psi=\widehat{w}$ for some $w\in W$. Moreover, since $\psi$ preserves $s_0$, it follows that $w\in \C_W(s_0)$. Hence $\Inn(W)\cap \p\leq \big\langle \widehat{w}~\big|~w\in \C_W(s_0)\big\rangle$. Conversely, a straightforward check shows that 
\begin{equation}\label{Inn intersection P}
    \widehat{s_0}= \prod_{j\in I\setminus I_2} \phi_{j,{m_j-1}},
\end{equation}  
and, if there exists some $ j\in I_\even$, then \begin{equation}\label{Inn intersection C}
\widehat{\Delta_{j}}=\prod_{k\neq j}\psi_{j,k}.\end{equation}
Thus, by Lemma $\ref{structure of the centralizer for s0}$, we have $\big\langle \widehat{w}~\big|~w\in \C_W(s_0)\big\rangle\leq \Inn(W)\cap \p $. This shows the equality $\Inn(W)\cap\p =\big\langle \widehat{w}~\big|~w\in \C_W(s_0)\big\rangle$.

By Corollary $\ref{automorphisms that preserve the maximal standard parabolic subgroups}$, we have $\Inn(W)\cap \p_1\leq \langle \widehat{s_0}\rangle$ and by $\eqref{Inn intersection P}$ the equality follows.

When $n_2=0$, by Corollary $\ref{lemma for the C-automorphisms}$, $\Inn(W)\cap \p_2\leq \langle \widehat{\Delta_j}\mid j\in I_\even\rangle$. Further, the equality holds due to $\eqref{Inn intersection C}$. When $n_2\geq 1$, together with Corollary $\ref{lemma for the C-automorphisms}$, and the relations $\eqref{Inn intersection P}$, and $\eqref{Inn intersection C}$, we get $\Inn(W)\cap \p_2=\langle \widehat w \mid w\in \C_W(s_0)\rangle$.
\end{proof}

\begin{lemma}\label{Diagrams normalize the rest of the automorphisms}
The group of diagram automorphisms $\diag(W)$ normalize each of the subgroups $\T,\p_1$, and $\p_2$ of $\Aut(W)$.
\end{lemma}

\begin{proof}
 Let ${\rho}\in \diag(W)$. For some $i\in I_2$, evaluating ${ \rho}^{-1} \tau_i{\rho}$ on the generators of $W$, we obtain
\begin{equation}\label{diagrams normalize trasvections}
         { \rho}^{-1} \tau_i{\rho}(s_j)={ \rho}^{-1} \tau_i(s_{\rho(j)})=\begin{cases}
             s_i,&\text{if $\rho(j)\neq i$,}\\
             s_js_0,&\text{if $\rho(j)=i$.}
         \end{cases}
\end{equation} Thus, $ { \rho}^{-1} \tau_i{\rho}=\tau_{\rho^{-1}(i)}\in \T$. Let $\phi_{i,{t_i}}\in \p_{1,i}\leq \p_1$ where $1\leq k_i<m_i$ with $\gcd(k_i,m_i)=1$. Then \begin{equation}\label{graph maps normalizes phi maps}
         {\rho}^{-1}\phi_{i,{t_i}}{\rho}(s_j)={\rho}^{-1}\phi_{i,{t_i}}(s_{\rho(j)})=\begin{cases}
             s_j,&\text{if $\rho(j)\neq i$,}\\
             s_j(s_0s_j)^{{k}_i-1},&\text{if $\rho(j)=i$.}
         \end{cases}
\end{equation}
Thus, ${\rho}^{-1}\phi_{i^{t_i}}{\rho}=\phi_{{\rho^{-1}(i)},{t_i}}\in \p_{1,\rho^{-1}(i)}\leq \p_1$. Next, suppose $\psi_{i,j}\in \p_{2,i}\leq\p_2$ for some $1\leq j\leq n$ with $j\neq i$. Then \begin{equation}\label{graph maps normalizes psi maps}
         {\rho}^{-1}\psi_{i,j}{\rho}(s_k)={\rho}^{-1}\psi_{i,j}(s_{\rho(k)})=\begin{cases}
            s_k ,&\text{if $\rho(k)\neq j$,}\\
            \Delta_{\rho^{-1}(i)}s_k \Delta_{\rho^{-1}(i)},&\text{if $\rho(k)=j$.}
         \end{cases}
\end{equation}
Thus, ${\rho}^{-1}\psi_{i,j}{\rho}=\psi_{\rho^{-1}(i),\rho^{-1}(j)}\in \p_{2,\rho^{-1}(i)}\leq \p_2$. Similarly, when $i\in I_2$, for $\sigma_{i,j}\in\p_{2,i}\leq \p_2$, where $j\in I\setminus I_2$, \begin{equation}\label{graph maps normalizes sigma maps}
{\rho}^{-1}\sigma_{i,j}{\rho}=\sigma_{\rho^{-1}(i),\rho^{-1}(j)}\leq \p_{2, \rho^{-1}(i)}\leq \p_2.
\end{equation} This completes the proof.
\end{proof}

We are now ready to delve into the automorphism group of $W$.

\begin{theorem}\label{Automorphism group of star graph}
 The automorphism group $\Aut(W)$ of $W$ is generated by inner automorphisms, partial conjugations of type $1$ and type $2$, transvections, and diagram automorphisms. In particular, $\Aut(W)$ decompose as $$\Aut(W) = (( \Inn(W) \p) \rtimes \T)\rtimes \diag(W).$$
\end{theorem}

\begin{proof}
Let $\varphi\in \Aut(W)$. Due to Lemma $\ref{s0 maps to s0 under automorphism of W}$ and Lemma $\ref{maximal parabolic subgroups goes to its conjugate}$, we may assume, up to the action of $\Inn(W)$ and $\diag(W)$, that $\varphi$ satisfies the following conditions:
\begin{itemize}
\item[(i)] $\varphi(s_0) = s_0$.
\item[(ii)] For each $i \in I$, there exists $x_i \in W$ such that $\varphi(W_i) = W_i^{x_i}$.
\end{itemize}

Since $s_0 \in \varphi(W_i) = W_i^{x_i}$, there exists some $w_i \in W_i$ such that $s_0 = x_i w_i s_0 w_i^{-1} x_i^{-1}$. This implies that the product $x_i w_i$ lies in the centralizer $\C_W(s_0)$. By replacing each $x_i$ with the representative $\overline{x_i} = x_i w_i$, we can assume without loss of generality that $x_i \in \C_W(s_0)$ for all $i \in I$, noting that $W_i^{x_i} = W_i^{\overline{x_i}}$ remains invariant under this substitution. Consequently, for each $i \in I$, the image of the generator $s_i$ under $\varphi$ can be expressed as $\varphi(s_i) = x_i y_i x_i^{-1}$, where $y_i \in W_i$ is a reflection such that $\langle s_0, y_i \rangle = W_i$. Consider the map $\widetilde{\varphi}: S \to W$ defined by: $$ \widetilde{\varphi}(s_0) = s_0 \quad \text{and} \quad \widetilde{\varphi}(s_i) = y_i \text{ for each } i \in I.$$ This map naturally extends to an automorphism of $W$. Note that $\widetilde{\varphi}\in \T\times\p_1$ (due to Corollary \ref{automorphisms that preserve the maximal standard parabolic subgroups}) and leaves the centraliser $\C_{W}(s_0)$ of $s_0$ invariant (by Lemma \ref{Structure of P}, $\widetilde{\varphi}\in \T\times\p_1$). By replacing $\varphi$ with the composition $\widetilde{\varphi}^{-1}\varphi$, we may henceforth assume that $$\varphi(s_i) = z_is_iz_i^{-1}\; \text{for some}\; z_i \in \C_W(s_0) = \langle s_0 \rangle \times \Delta.$$ Next, we shift each $z_i$ into $\Delta$. Let $i$ be an index such that $z_i \notin \Delta$. Then $z_i = s_0 \overline{z_i}$ for some $\overline{z_i} \in \Delta$. Consider the partial conjugation $\phi_{i,{m_i-1}}$. By definition $\phi_{i,{m_i-1}}  \varphi (s_i) = \overline{z_i}s_i\overline{z_i}^{-1},\; \text{and}\; \phi_{i,{m_i-1}}  \varphi (s_j) = \varphi (s_j),$ for all other generators $s_j$. Repeating this substitution for each such $i$ (for which $z_i \notin \Delta$), we ultimately obtain: $$\varphi(s_i) = z_is_iz_i^{-1}\; \text{for some}\; z_i \in  \Delta.$$ Finally, due to Corollary $\ref{lemma for the C-automorphisms}$, we have $\varphi \in \p_2$, proving that $\Aut(W)$ is generated by $\Inn(W)$, $\p_1$, $\p_2$, $\T$, and $\diag(W)$.

To establish the structural decomposition of $\Aut (W)$, we first note that by Lemma \ref{structure of P and C}, the subgroup generated by $\Inn(W)$ and $\p$ is given by the product $\Inn(W) \p$. By Lemma $\ref{inner intersection P}$, $\T$ normalizes this product and intersects it trivially. Therefore, we obtain the semi-direct product structure $N = \langle \Inn(W) \p, ~ \T \rangle = ( \Inn(W)\p ) \rtimes\T$. Furthermore, $N$ is a normal subgroup of $\Aut (W)$ as it is normalized by $\diag(W)$ by the Lemma $\ref{Diagrams normalize the rest of the automorphisms}$. Finally, we observe that the intersection $N \cap \diag(W)$ is trivial. Indeed, elements of $\diag(W)$ act by permuting the set of maximal standard parabolic subgroups $\{W_i\}_{i \in I}$, whereas every element of $N$ maps each $W_i$ to a conjugate subgroup $W_i^{g_i}$ for some $g_i \in W$. Since no two distinct maximal standard parabolic subgroups are conjugate (Lemma \ref{maximal parabolic subgroups are not conjugate}), an automorphism that both permutes these subgroups and maps them to conjugates must necessarily fix each $W_i$. The only diagram automorphism satisfying this condition is the identity. Thus, $N \cap \diag(W) = \{1\}$, establishing the desired semi-direct product $\Aut(W) = ((\Inn(W) \p) \rtimes \T)\rtimes \diag(W)$.
\end{proof}

As a corollary, we determine the special automorphisms group $\Spe(W)$ of $W$.

\begin{corollary}
The special automorphism group $\Spe(W)$ of $W$ decomposes as $$\Spe(W) = (\Inn(W)\p)\rtimes \diag_o.$$ Consequently, the full automorphism group decomposes as $$\Aut(W)=\Spe(W)\rtimes (\T\rtimes \diag_e).$$ In particular, $\Aut(W) = \Spe(W)$ if and only if $n_2=0$ and all the even labels have multiplicity one.
\end{corollary}

\begin{proof}
From the definitions of the inner automorphisms, the transvections, and the partial conjugations, we have $\T\cap \Spe(W)=1$ and $(\Inn(W))\p\leq \Spe(W)$. By Lemma $\ref{conjugacy lemma through odd graphs}$, the generators $s_i$ for all $i\in I_\odd$ belong to a single conjugacy class. Consequently, the diagram automorphisms of $\diag_o$ preserve the conjugacy class of the involutions. Then, by Theorem $\ref{Automorphism group of star graph}$, we get $\Spe(W)=(\Inn(W)\p)\rtimes \diag_o$. Moreover, $\diag_o$ commutes with both $\diag_e$ and $\T$ as its action is restricted to the vertices corresponding to the odd labels. Thus, we obtain the following splitting $\Aut(W) = \Spe(W)\rtimes (\T\times \diag_e)$.
\end{proof}

\begin{remark}
The complete structural characterization of $\Aut (W)$ is established by the descriptions of its constituent factors: the center-less property of $W$ yields  $\Inn (W) \cong W$ (Lemma \ref{W has trivial center}); the group $\p = \p_1 \times_{\p3} \p_2$ is determined by Lemmas \ref{Structure of P}, \ref{structure of C}, \ref{structure of P and C}; $\T \cong (\mathbb{Z}_2)^{n_2}$ follows from Lemma \ref{Structure of T}); and $\diag (W)$ is determined by the symmetries of the Coxeter diagram (Lemma \ref{structure of diagonal automorphisms}).
\end{remark}

Next, we discuss the natural short exact sequence $1 \to \Inn (W) \to \Aut (W) \to \Out (W) \to 1$. We begin with the following elementary number-theoretic observation.
 
\begin{lemma}\label{lemma for the splitting of unit modulo groups}
Let $m_1,m_2,\ldots, m_l > 1$ be some integers. Then the short exact sequence $1\to \mathbb{Z}_2\to \prod_{i=1}^l \mathbb{Z}^\times _{m_i}\to \Big(\prod_{i=1}^{l}\mathbb{Z}^\times _{m_i}\Big)/\mathbb{Z}_2\to 1$ splits if and only if there exists some index $1\leq i\leq l$ such that $m_i$ is divisible by either $4$ or a prime $p\equiv 3~(\modulo~4)$.
\end{lemma}

\begin{lemma}\label{theorem-inner splits with non-trivial automorphisms}
The short exact sequence \begin{equation}\label{inner splits with the non-trivial automorphisms}
        1\to \Inn(W)\to \Inn(W)\p\to \Inn(W)\p/\Inn(W)\to 1
    \end{equation} splits if and only if there exists some index $i\in I$ such that $m_i$ is divisible by either $4$ or a prime $p\equiv 3$ $(\modulo~4)$.
\end{lemma}

\begin{proof}
Suppose the short exact sequence \eqref{inner splits with the non-trivial automorphisms} splits. Since $\p_1\cong \prod_{i\in I\setminus I_2} \mathbb{Z}_{m_i}^\times$ and $\Inn(W)\cap \p_1 = \langle \widehat{s_0}\rangle \cong \mathbb{Z}_2$, it follows that the sequence $1 \to \langle \widehat{s_0}\rangle \to \p_1 \to \p_1/\langle \widehat{s_0}\rangle \to 1$ also splits. Then, by the Lemma $\ref{lemma for the splitting of unit modulo groups}$, there exists some index $i\in I$ such that $m_i$ satisfies the given divisibility condition.

Conversely, assume there exists $i\in I$ such that $m_i$ is divisible by $4$ or a prime $p\equiv 3~(\modulo~4)$. Without loss of generality, let $i=1$. We construct a complement $\Q$ of $\Inn (W)$ in $\Inn (W) \p$ by analysing $\p_1$ and $\p_2$ separately.

Since $m_1$ is divisible by either $4$ or a prime $p\equiv 3$ $(\modulo~4)$, the abelian group $\p_{1,1} = \{ \phi_{1,t} \mid 1\leq t \leq m_1 - 1, \gcd (t, m_1) = 1\}  \cong \mathbb{Z}_{m_1}^\times$ contains $\langle \phi_{1,{m_1-1}} \rangle \cong \mathbb{Z}_2$ as a direct factor. Let $\Q_{1,1}$ be a complement of $\langle \phi_{1,{m_1-1}} \rangle$ in $\p_{1,1}$. Define $\Q_1 := \Q_{1,1}\times \prod_{j\neq 1}\p_{1,j}$. We then have $\p_1= \langle \phi_{1,{m_1-1}}\rangle \times \Q_1.$ Since $\phi_{1,{m_1-1}} =\widehat{s_0}\prod_{j\neq 1}\phi_{j,{m_j-1}}$, it follows that $\phi_{1,{m_1-1}} \in \Inn(W)\Q_1$. This inclusion implies the equality $\Inn(W)\p_1=\Inn(W)\Q_1$.  

Similarly, for each $j \in I_\even$, we decompose $\p_{2,j}$ by defining the subgroups $\R_{2,j}$ and their complements $\Q_{2,j}$ as follows:
\begin{itemize}
\item If $j=1 \in I_\even$, let $\R_{2,1} = \langle \psi_{1,2} \rangle$ and $\Q_{2,1} = \langle \psi_{1,k} \mid k \neq 1,2 \rangle$.
\item If $j \in I_\even \setminus (I_2 \cup \{1\})$, let $\R_{2,j} = \langle \psi_{j,1} \rangle$ and $\Q_{2,j} = \langle \psi_{j,k} \mid k \neq 1,j \rangle$.
\item If $j \in I_2$, let $\R_{2,j} = \langle \psi_{j,1}, \sigma_{j,1} \rangle$ and $\Q_{2,j} = \langle \psi_{j,k}, \sigma_{j,k} \mid k \neq 1,j \rangle$.
\end{itemize}
In each case, we have $\p_{2,j} = \R_{2,j} \times \Q_{2,j}$ and $\R_{2,j} \leq \Inn (W)\Q_{2,j}$. Let $\Q_2 = \langle \Q_{2,j} \mid j \in I_\even \rangle$. Then $\Inn (W)\p_2 = \Inn (W)\Q_2$.

Let $\Q$ be the subgroup generated by $\Q_1$ and $\Q_2$. By construction, $\Inn (W) \p = \Inn (W) \Q$. To establish the splitting, we show that $\Inn (W) \cap \Q = 1$. Any automorphism $\varphi \in \Inn(W) \cap \Q$ is of the form $\widehat{w}$ for some $w \in \C_W(s_0)$. Since $\Q_2$ fixes $s_0$, and by the construction of $\Q_1$, we have $$\varphi (s_1) = s_1^{w},\; \text{where}\; w \in W_1 \cap \C_W(s_0)=\begin{cases}
    \langle s_0\rangle, &\text{if}~ 1\in I_\odd,\\
    \langle s_0\rangle \times \langle \Delta_1\rangle, & \text{if}~1\in I_\even.
\end{cases}$$

If $w = s_0$, then the relation $\phi_{1, m_1-1} =\widehat{s_0}\prod_{i\neq 1}\phi_{i, m_i-1}$ would imply $\phi_{1, m_1-1}\in \Q$, which contradicts the fact that $\Q_{1,1}$ is a complement to $\langle \phi_{1, m_1-1} \rangle$. Similarly, if $1\in I_\even$, and $w\in \{\Delta_{1}, s_0\Delta_1\}$, then we can write $\phi_{1, m_1-1}$ as
\begin{equation*}
    \phi_{1, m_1-1}=\begin{cases}
        \widehat{s_0}\prod_{i\neq 1}\psi_{1,i}, & \text{if}~w=\Delta_1,\\
        \widehat{s_0}\prod_{i\neq 1}\phi_i^{m_i-1}\psi_{1,i}, & \text{if}~w=s_0\Delta_1.
    \end{cases}
\end{equation*}
In both cases, this would imply $\phi_{1, m_1-1} \in \Q$, which is again a contradiction. Thus, we must have $w=1$, which implies $\Inn(W) \cap \Q = 1$. Hence, $\Inn(W)\p = \Inn(W)\rtimes \Q$, and we conclude that the sequence \eqref{inner splits with the non-trivial automorphisms}.
\end{proof}

\begin{theorem}\label{Splitting-of-AutW}
    The short exact sequence
  \begin{equation}\label{inner splits with outer}
        1\to \Inn(W)\to \Aut(W)\to \Out(W)\to 1
    \end{equation} splits if there exists some index $i\in I$ such that $m_i$ is divisible by either $4$ or a prime $p\equiv 3$ $(\modulo~4)$ and multiplicity of the label $m_i$ is $1$. In particular, $$\Out(W) = (\Q\rtimes \T )\rtimes\diag(W).$$
\end{theorem}

\begin{proof} 
Suppose that there exists some $i\in I$ such that the label $m_i$ satisfies the given hypothesis. Without loss of generality, let $i=1$. 
By the Lemma $\ref{theorem-inner splits with non-trivial automorphisms}$, the short exact sequence \eqref{inner splits with the non-trivial automorphisms} splits, thus we have $\Inn(W)\p=\Inn(W)\rtimes \Q$. Furthermore Theorem \ref{Automorphism group of star graph}, implies that the automorphism group admits the decomposition $\Aut(W)=((\Inn(W)\rtimes \Q)\rtimes \T)\rtimes \diag(W)$. We next show that the subgroup $H := \langle \Q, \T, \diag(W) \rangle$ is a complement to $\Inn(W)$ in $\Aut (W)$.

From the proof of the Lemma $\ref{structure of T P C}$, it can be easily verified that $\T$ normalizes $\Q$ with $\Q\cap \T=1$, yielding the semi-direct product $\langle \Q, \T\rangle =\Q\rtimes \T$. Additionally, Lemma $\ref{Diagrams normalize the rest of the automorphisms}$ ensures that $\diag(W)$ normalizes $\Q\rtimes \T$. Since $\Q\rtimes \T\leq \p\rtimes \T = 1$ by Theorem $\ref{Automorphism group of star graph}$, it follows that $H =(\Q\rtimes \T)\rtimes \diag(W)$.

To prove that the sequence \eqref{inner splits with outer} splits, it remains to show that $\Inn (W) \cap H = 1$. Assume that $\widehat{w}\in \Inn(W)\cap H$. Then $\widehat{w}=\phi \tau \rho$ for some $\phi\in \Q$, $\tau\in \T$, and $\rho\in \diag(W)$. Since every element of $H$ fixes the central vertex $s_0$, we have $\widehat{w}(s_0)=ws_0w^{-1}=s_0$, i.e., $w\in \C_W(s_0)$. Since the index $1\notin I_2$ and the label $m_1$ occurs with multiplicity $1$, $\phi \tau \rho(s_1)=\phi(s_1)=ws_1w^{-1}$. Since $\Q_2$ also fixes generator $s_1$, we have $w\in W_1$. Thus, $w\in \C_W(s_0)\cap W_1$.

As shown in the proof of Lemma \ref{theorem-inner splits with non-trivial automorphisms}, the condition $w \in \C_W(s_0)\cap W_1$ forces $w=1$ under the given divisibility constraints on $m_1$. Thus, $\Inn(W)\cap H=1.$ Consequently, $\Aut(W)=\langle\Inn(W), \Q, \T, \diag(W)\rangle= \langle \Inn(W), H\rangle=\Inn(W)\rtimes H$, i.e., the short exact sequence $1\to \Inn(W)\to \Aut(W)\to \Out(W)\to 1$. In particular, $\Out(W)=H=(\Q\rtimes \T)\rtimes \diag(W)$.
\end{proof}

The findings in Theorem $\ref{Automorphism group of star graph}$ and Theorem $\ref{Splitting-of-AutW}$ recover, as a special case, the known description of the automorphism group of odd Coxeter groups whose finite diagram is a tree \cite[Theorem 3.6]{NS-21}, along with the corresponding splitting criteria \cite[Theorem 4.5]{NS-21}. We conclude this section with the $R_\infty$ property for groups $W \in \sw$.

\begin{theorem}\label{R-infinity Property}
All Coxeter groups $W\in \sw$ satisfy the $R_\infty$-property.
\end{theorem}

\begin{proof}
Let $W\in \sw$ be a Coxeter group with the generating set $S$. Note that it does not contain any parabolic subgroups which is an affine Coxeter group. Furthermore, for any disjoint subsets $T_1, T_2 \subseteq S$ such that the corresponding parabolic subgroups $W_{T_1}$ and $W_{T_2}$ are infinite, there necessarily exist vertices $s \in T_1$ and $s' \in T_2$ with $m(s, s') = \infty$. This ensures that $W_{T_1}$ and $W_{T_2}$ never commute. Consequently, by Moussong's characterization \cite[Theorem 17.1]{Mo-88}, $W$ is a hyperbolic group. Since $W$ additionally contains a free subgroup of rank 2 (e.g., the subgroup generated by $s_is_j$ and $s_0s_is_js_0$ for $i\in I\setminus I_2$), it is a non-elementary hyperbolic group. Applying the result of Fel'shtyn \cite[Theorem 3]{Fe-04}, which states that the Reidemeister number of any automorphism of a non-elementary Gromov hyperbolic group is infinite, we conclude that $W$ possesses the $R_\infty$-property.
\end{proof}

\section{Rigidity}\label{Rigidity}
Brady, M\"uhlherr, McCammond, and Neumann introduced diagram twisting in \cite{BMMN-02}. 

\begin{definition}
    Let $(W, S)$ be a Coxeter system and $I, J\subseteq S$. The pair $(I,J)$ is called {\it $S$-admissible} if \begin{enumerate}
    \item[1.] $W_I$ is a finite standard parabolic subgroup, and $(I\cup I^{\perp})\cap J=\emptyset$, where $I^{\perp}:=\{s\in S \mid m(s,s_i)=2,~\text{for all~}s_i\in I\}$.
    \item[2.] For all $s_j\in J$ and $s_k\in K_{I,J}:=S\setminus( I\cup I^{\perp}\cup J)$, we have $m(s_j,s_k)=\infty$.
\end{enumerate} Then $\T_{I,J}(S):=  I\cup I^{\perp}\cup J\cup \{\Delta_{I}s_k\Delta_{I} \mid s_k\in K_{I,J}\}$
is a Coxeter generating set of $W$ which is contained in $S^W$ and it is called a twist of $S$ in $W_I$. 
\end{definition}

A twist is called {\it trivial} if either $J$ or $K_{I, J}$ is empty. The finite diagram $\mathcal{V}_{(W, S)}$ is isomorphic to $\mathcal{V}_{(W,\T_{I,J}(S))}$ if $\Delta_I\in \Z(W_I)$.

For instance, consider the following example illustrating the diagram twist (see \cite{Mu-00}).
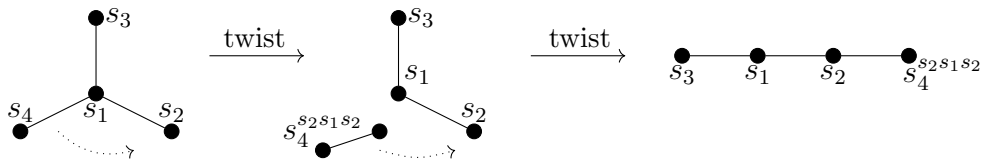
\begin{figure}[ht]
    \centering
  \begin{tikzpicture}
	\begin{pgfonlayer}{nodelayer}
		\node [style=new style 0, color=black, scale=0.5] (0) at (1, 1.25) {};
		\node [style=new style 0, color=black, scale=0.5] (1) at (1, 0.25) {};
		\node [style=new style 0, color=black, scale=0.5] (2) at (0, -0.25) {};
		\node [style=new style 0, color=black, scale=0.5] (3) at (2, -0.25) {};
		\node [style=none] (4) at (1, 0) {$s_1$};
		\node [style=none] (5) at (0, 0) {$s_4$};
		\node [style=none] (6) at (2, 0) {$s_2$};
		\node [style=none] (7) at (1.3, 1.25) {$s_3$};
		\node [style=none] (8) at (2.5, 0.75) {};
		\node [style=none] (9) at (3.75, 0.75) {};
		\node [style=none] (10) at (3.1, 1) {\small{twist}};
		\node [style=none] (11) at (0.5, -0.25) {};
		\node [style=none] (12) at (1.5, -0.5) {};
		\node [style=new style 0, color=black, scale=0.5] (13) at (5, 1.25) {};
		\node [style=new style 0, color=black, scale=0.5] (14) at (5, 0.25) {};
		\node [style=new style 0, color=black, scale=0.5] (15) at (4, -0.5) {};
		\node [style=new style 0, color=black, scale=0.5] (16) at (6, -0.25) {};
		\node [style=new style 0, color=black, scale=0.5] (17) at (4.75, -0.25) {};
		\node [style=none] (18) at (4.75, -0.5) {};
		\node [style=none] (19) at (5.75, -0.5) {};
		\node [style=none] (20) at (5.3, 1.25) {$s_3$};
		\node [style=none] (21) at (5.25, 0.5) {$s_1$};
		\node [style=none] (22) at (6, 0) {$s_2$};
		\node [style=none] (23) at (4, -0.25) {$s_4^{s_2s_1s_2}$};
		\node [style=none] (24) at (6.75, 0.75) {};
		\node [style=none] (25) at (8, 0.75) {};
		\node [style=none] (26) at (7.4, 1) {\small{twist}};
		\node [style=new style 0, color=black, scale=0.5] (27) at (8.75, 0.75) {};
		\node [style=new style 0, color=black, scale=0.5] (28) at (9.75, 0.75) {};
		\node [style=new style 0, color=black, scale=0.5] (29) at (10.75, 0.75) {};
		\node [style=new style 0, color=black, scale=0.5] (30) at (11.75, 0.75) {};
		\node [style=none] (31) at (8.75, 0.5) {$s_3$};
		\node [style=none] (32) at (9.75, 0.5) {$s_1$};
		\node [style=none] (33) at (10.75, 0.5) {$s_2$};
		\node [style=none] (34) at (12.2, 0.5) {$s_4^{s_2s_1s_2}$};
	\end{pgfonlayer}
	\begin{pgfonlayer}{edgelayer}
		\draw (0) to (1);
		\draw (1) to (2);
		\draw (1) to (3);
		\draw [style=new edge style 1] (8.center) to (9.center);
		\draw[dotted] [style=new edge style 1, bend right] (11.center) to (12.center);
		\draw (13) to (14);
		\draw (14) to (16);
		\draw (15) to (17);
		\draw[dotted] [style=new edge style 1, bend right=15, looseness=1.25] (18.center) to (19.center);
		\draw [style=new edge style 1] (24.center) to (25.center);
		\draw (27) to (28);
		\draw (28) to (29);
		\draw (29) to (30);
	\end{pgfonlayer}
\end{tikzpicture}
\end{figure}

Let us consider the following two Coxeter systems, given as $W_1=\big\langle s_1,s_2,s_3,s_4 \mid s_i^2=1=(s_1s_j)^3,~1\leq i\leq 4,~2\leq j\leq 4\big\rangle$, and $W_2=\big\langle t_1,t_2,t_3,t_4 \mid t_i^2=1=(t_1t_j)^{3}=(t_2t_4)^3,~1\leq i\leq 4,~2\leq j\leq3\big\rangle$.
Denote $S=\{s_i\mid 1\leq i\leq 4\}$ and $T=\{t_i\mid 1\leq i\leq 4\}$. Take $I=\{s_1,s_2\}$ and $J=\{s_3\}$. Then $W_I$ is finite. Further, $I^{\perp}=\emptyset$ and $K_{I,J}=S\setminus \{s_1,s_2,s_3\}=\{s_4\}$, where $m(s_3,s_4)=\infty$. Therefore, $(I,J)$ is an $S$-admissible pair and hence $\T_{I,J}(S)=\{s_1,s_2,s_3,s_2s_1s_2s_4s_2s_1s_2\}$ is a twist of $S$ in the dihedral group $W_I\cong D_6$. Moreover, the map $\varphi: \T_{I,J}(S)\to T$, defined by $\varphi(s_i)=t_i~\text{for $1\leq i\leq 3$,}~\text{and}~\varphi(s_2s_1s_2s_4s_2s_1s_2)=t_4,$ is an isomorphism from $W_1$ onto $W_2$.
\medskip

The concept of the pseudo-transpositions was introduced by Mihalik in \cite{Mi-07}. 

\begin{definition}
Let $(W, S)$ be a Coxeter system. An element $\tau\in S$ is a pseudo-transposition if
\begin{enumerate}
\item there exists a unique $t\in S$ with $m(\tau,t)=2\pmod 4$;
\item for each $s\in S\setminus\{\tau,t\}$, one has $m(s,\tau)\in\{2,\infty\}$, and if $m(s,\tau)=2$ then $m(t,s)=2$.
\end{enumerate}
If $\tau\in S$ is a pseudo-transposition, then $(S\setminus\{\tau\})\cup\{\tau t\tau, \triangle_{\tau t}\}$ is also a Coxeter generating set of $W$.     
\end{definition}

This process of tweaking the Coxeter generating set by replacing a pseudo-transposition is referred to as {\it blowing up}. Note that the rank increases after each blowing up. We refer to the reverse process as {\it triangle elimination}.

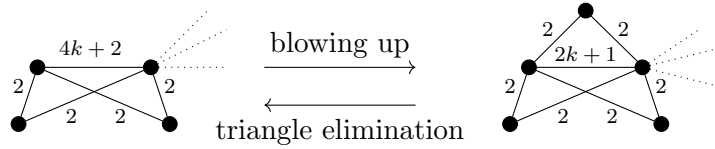
\begin{figure}[ht]
\centering
\begin{tikzpicture}
	\begin{pgfonlayer}{nodelayer}
		\node [style=new style 0, scale=0.5, color=black] (0) at (2, 0) {};
		\node [style=new style 0, scale=0.5, color=black] (1) at (3.5, 0) {};
		\node [style=new style 0, scale=0.5, color=black] (2) at (1.75, -0.75) {};
		\node [style=new style 0, scale=0.5, color=black] (3) at (3.75, -0.75) {};
		\node [style=none] (4) at (4.25, 0.75) {};
		\node [style=none] (5) at (4.5, 0.5) {};
		\node [style=none] (6) at (4.5, 0) {};
		\node [style=none] (7) at (2.7, 0.25) {\tiny{$4k+2$}};
		\node [style=none] (8) at (5, 0) {};
		\node [style=none] (9) at (7, 0) {};
		\node [style=none] (10) at (6, 0.3) {\small{blowing up}};
		\node [style=none] (11) at (5, -0.5) {};
		\node [style=none] (12) at (7, -0.5) {};
		\node [style=none] (13) at (6, -0.85) {\small{triangle elimination}};
		\node [style=new style 0, scale=0.5, color=black] (14) at (8.5, 0) {};
		\node [style=new style 0, scale=0.5, color=black] (15) at (10, 0) {};
		\node [style=new style 0, scale=0.5, color=black] (16) at (8.25, -0.75) {};
		\node [style=new style 0, scale=0.5, color=black] (17) at (10.25, -0.75) {};
		\node [style=none] (18) at (10.75, 0.5) {};
		\node [style=none] (19) at (11, 0.25) {};
		\node [style=none] (20) at (11, -0.25) {};
		\node [style=none] (21) at (9.25, 0.17) {\tiny{$2k+1$}};
		\node [style=new style 0, scale=0.5, color=black] (22) at (9.25, 0.75) {};
		\node [style=none] (23) at (1.75, -0.25) {\tiny{$2$}};
		\node [style=none] (24) at (2.45, -0.65) {\tiny{$2$}};
		\node [style=none] (25) at (3.1, -0.65) {\tiny{$2$}};
		\node [style=none] (26) at (3.75, -0.25) {\tiny{$2$}};
		\node [style=none] (27) at (8.25, -0.25) {\tiny{$2$}};
		\node [style=none] (28) at (8.95, -0.65) {\tiny{$2$}};
		\node [style=none] (29) at (9.6, -0.65) {\tiny{$2$}};
		\node [style=none] (30) at (10.25, -0.25) {\tiny{$2$}};
		\node [style=none] (31) at (8.75, 0.5) {\tiny{$2$}};
		\node [style=none] (32) at (9.75, 0.5) {\tiny{$2$}};
	\end{pgfonlayer}
	\begin{pgfonlayer}{edgelayer}
		\draw (0) to (1);
		\draw (0) to (2);
		\draw (1) to (2);
		\draw (0) to (3);
		\draw (1) to (3);
		\draw [dotted] (1) to (6.center);
		\draw [dotted] (1) to (5.center);
		\draw [dotted] (1) to (4.center);
		\draw [style=new edge style 1] (8.center) to (9.center);
		\draw [style=new edge style 1] (12.center) to (11.center);
		\draw (14) to (15);
		\draw (14) to (16);
		\draw (15) to (16);
		\draw (14) to (17);
		\draw (15) to (17);
		\draw [dotted] (15) to (20.center);
		\draw [dotted] (15) to (19.center);
		\draw [dotted] (15) to (18.center);
		\draw (14) to (22);
		\draw (15) to (22);
	\end{pgfonlayer}
\end{tikzpicture}
\end{figure}

Now we recall the matching theorems for finitely generated Coxeter groups that provide intrinsic structural information about the maximal finite parabolic subgroups of the Coxeter groups, in the sense that they are almost invariant under the choice of the Coxeter generators for the Coxeter groups. Mihalik, Ratcliffe, and Tschantz have discussed the matching theorems in detail in \cite{MRT-07}.

\begin{definition}
A subset $C\subseteq S$ is a {\it simplex} of $(W, S)$ if for all $s,t\in C$ one has $m(s,t)<\infty$.
A simplex is \emph{spherical} if $\langle C\rangle$ is finite. Let $(W, S)$ be a Coxeter system.
\end{definition}  

\begin{definition}
A subset $B\subseteq S$ is called {\it basic} if it is irreducible, maximal, and $\langle B\rangle$ is a non-cyclic finite subgroup. If $B$ is basic, then $\langle B\rangle$ is a {\it basic subgroup} of $W$.
\end{definition}

\begin{lemma}\cite[Proposition 4.13]{MRT-07}\label{4.13[MRT07]}
Let $W$ be a finitely generated Coxeter group with two sets of Coxeter generators $S$ and $S'$, and let $M$ be a maximal spherical simplex of $(W, S)$. Then there is a unique maximal spherical simplex $M'$ of $(W, S')$ such that $\langle M\rangle $ and $\langle M'\rangle $ are conjugate in $W$.
\end{lemma}

Note that, for the Coxeter systems $(W, S)$ where $\mathcal{V}_{(W, S)}$ is a star diagram, the subgroups generated by the maximal spherical simplices are precisely the maximal finite standard parabolic subgroups. Apart from the maximal simplices of type $\mathbb{Z}_2\times \mathbb{Z}_2$, the rest of the maximal simplices are precisely basic subsets of $S$.

The following matching results will be helpful to figure out the correspondence between the maximal simplices for any two Coxeter generators of $W\in \sw$.

\begin{lemma} \cite[Theorem 4.18]{MRT-07} \label{Basic Matching Theorem}
Let $W$ be a finitely generated Coxeter group with Coxeter generating sets $S$ and $S'$, and let $B$ be a basic subset of $S$. Then there exists a unique irreducible subset $B'$ of $S'$ such that
$\langle B\rangle'$ is conjugate to $\langle B'\rangle'$ in $W$.
Moreover
\begin{enumerate}
\item the set $B'$ is a basic subset of $(W, S')$, and we say that $B$ and $B'$ match;
\item if $|\langle B\rangle|=|\langle B'\rangle|$, then $B$ and $B'$ have the same type and there is an isomorphism $\phi:\langle B\rangle \rightarrow \langle B'\rangle$ that restricts to conjugation on $\langle B\rangle'$ by an element of $W$ and  we say that $B$ and $B'$ match isomorphically; 
\item if $|\langle B\rangle|<|\langle B'\rangle|$ and $\langle B\rangle$ has type $D_{2(2q+1)}$, then $\langle B'\rangle$ has type $D_{2(4q+2)}$. Moreover, there is a monomorphism $\phi:\langle B\rangle \rightarrow \langle B'\rangle$ that restricts to conjugation on $\langle B\rangle'$ by an element of $W$.
\end{enumerate}
\end{lemma}

Note that the preceding lemma is a part of a general result called the Basic Matching Theorem. 

\begin{theorem} \cite[Simplex Matching Theorem]{MRT-07} \label{Simplex Matching Theorem}
If $W$ is a finitely generated Coxeter group with two sets of Coxeter generators $S$ and $S'$ whose basic subsets match isomorphically, then there is a bijection between the simplices of $(W, S)$ and the simplices $(W, S')$ so that the corresponding simplices generate isomorphic Coxeter systems. In particular, $|S| = |S'|$, and the presentation diagrams of $(W, S)$ and $(W, S')$ have the same multiset of edge labels.
\end{theorem}

Based upon the results described above, we show that the Coxeter groups in $\sw$ are rigid up to diagram twisting and blowing up (or triangle elimination).

\begin{lemma}\label{structure of tree diagrams for groups in W(T)}
Let $W\in \sw$ with generating set $S$, and $S'$ be another Coxeter generating set for $W$. If the finite diagram $\mathcal{V}_{(W, S')}$ is a tree where each even labelled edge is incident to a vertex of degree one, then $\mathcal{V}_{(W, S')}\cong \mathcal{V}_{(W, S)}$, up to diagram twisting.
\end{lemma}

\begin{proof}
Choose a vertex $s_0$ which is incident to an odd edge, and fix it as the root of the tree. Consider an edge $\{s,t\}\in \mathcal{V}_{(W, S')}$ with $\degree(t)=1$ and $\degree(s)>1$. We may assume $s\neq s_0$. Then, by assumption, there exists a unique odd path in $\mathcal{V}_{(W, S)}$ from $s_0$ to $s$.
 
Denote this path by $P^\odd_{s_0,s}:=\{s_0,s_{i_1},\ldots,s_{i_k}=s\}$ for some integer $k\geq 1$. Set $I=\{s_{i_{k-1}},s\}$ and $J=S\setminus\{t, s, s_{i_{k-1}}\}$ and $K_{I,J}=\{t\}$. If $k=1$ then, we interpret $s_{i_0}$ as $s_0$. It follows that $(I, J)$ is an $S'$-admissible pair. Let $S_1'=\T_{I,J}(S')$. Then, in $\mathcal{V}_{(W, S'_1)}$, the vertex $t^{\Delta_{k-1,l}}$ becomes incident to $s_{i_{k-1}}$ instead of $s_k$ with the same edge label $m(s,t)$, where $\Delta_{j-1,j}=s_{j}(s_{j-1}s_j)^{\frac{m(s_j,s_{j-1})-1}{2}}$ for $1\leq j\leq k$. 
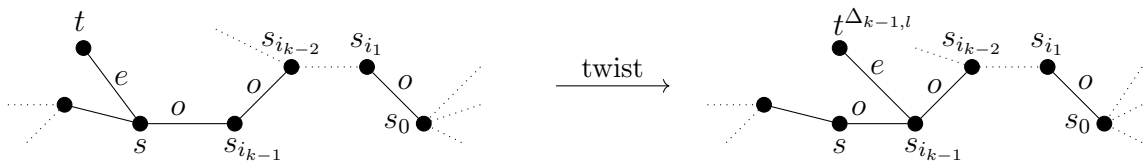
\begin{figure}[ht]
    \centering
\begin{tikzpicture}
	\begin{pgfonlayer}{nodelayer}
		\node [style=new style 0, scale=0.5, color=black] (0) at (11, 1.5) {};
		\node [style=new style 0, scale=0.5, color=black] (1) at (11, 0.5) {};
		\node [style=new style 0, scale=0.5, color=black] (2) at (10, 0.75) {};
		\node [style=new style 0, scale=0.5, color=black] (3) at (12, 0.5) {};
		\node [style=new style 0, scale=0.5, color=black] (4) at (12.75, 1.25) {};
		\node [style=new style 0, scale=0.5, color=black] (5) at (13.75, 1.25) {};
		\node [style=new style 0, scale=0.5, color=black] (6) at (14.5, 0.5) {};
		\node [style=none] (7) at (9.25, 0.75) {};
		\node [style=none] (8) at (9.5, 0.25) {};
		\node [style=none] (9) at (12, 1.5) {};
		\node [style=none] (10) at (15, 1.25) {};
		\node [style=none] (11) at (15, 0.75) {};
		\node [style=none] (12) at (15, 0.25) {};
		\node [style=none] (13) at (11.45, 1.85) {$t^{\Delta_{k-1,l}}$};
		\node [style=none] (14) at (11, 0.2) {$s$};
		\node [style=none] (15) at (12.25, 0.15) {$s_{i_{k-1}}$};
		\node [style=none] (16) at (12.75, 1.55) {$s_{i_{k-2}}$};
		\node [style=none] (17) at (13.75, 1.55) {$s_{i_1}$};
		\node [style=none] (18) at (14.15, 0.5) {$s_0$};
		\node [style=none] (19) at (11.5, 1.25) {$e$};
		\node [style=none] (20) at (11.25, 0.7) {$o$};
		\node [style=none] (21) at (12.25, 1) {$o$};
		\node [style=none] (22) at (14.25, 1.05) {$o$};
		\node [style=new style 0, scale=0.5, color=black] (23) at (1, 1.5) {};
		\node [style=new style 0, scale=0.5, color=black] (24) at (1.75, 0.5) {};
		\node [style=new style 0, scale=0.5, color=black] (25) at (0.75, 0.75) {};
		\node [style=new style 0, scale=0.5, color=black] (26) at (3, 0.5) {};
		\node [style=new style 0, scale=0.5, color=black] (27) at (3.75, 1.25) {};
		\node [style=new style 0, scale=0.5, color=black] (28) at (4.75, 1.25) {};
		\node [style=new style 0, scale=0.5, color=black] (29) at (5.5, 0.5) {};
		\node [style=none] (30) at (0, 0.75) {};
		\node [style=none] (31) at (0.25, 0.25) {};
		\node [style=none] (32) at (2.75, 1.75) {};
		\node [style=none] (33) at (6.25, 1.25) {};
		\node [style=none] (34) at (6.25, 0.75) {};
		\node [style=none] (35) at (6, 0.25) {};
		\node [style=none] (36) at (0.95, 1.85) {$t$};
		\node [style=none] (37) at (1.75, 0.2) {$s$};
		\node [style=none] (38) at (3.25, 0.15) {$s_{i_{k-1}}$};
		\node [style=none] (39) at (3.75, 1.55) {$s_{i_{k-2}}$};
		\node [style=none] (40) at (4.75, 1.55) {$s_{i_1}$};
		\node [style=none] (41) at (5.15, 0.5) {$s_0$};
		\node [style=none] (42) at (1.5, 1.1) {$e$};
		\node [style=none] (43) at (2.25, 0.7) {$o$};
		\node [style=none] (44) at (3.25, 1) {$o$};
		\node [style=none] (45) at (5.25, 1.05) {$o$};
		\node [style=none] (46) at (7.25, 1) {};
		\node [style=none] (47) at (8.75, 1) {};
		\node [style=none] (48) at (8, 1.25) {\small{twist}};
	\end{pgfonlayer}
	\begin{pgfonlayer}{edgelayer}
		\draw (2) to (1);
		\draw (1) to (3);
		\draw (3) to (4);
		\draw (5) to (6);
		\draw [dotted] (7.center) to (2);
		\draw [dotted] (8.center) to (2);
		\draw [dotted] (9.center) to (4);
		\draw [dotted] (10.center) to (6);
		\draw [dotted] (11.center) to (6);
		\draw [dotted] (12.center) to (6);
		\draw [dotted] (4) to (5);
		\draw (0) to (3);
		\draw (25) to (24);
		\draw (24) to (26);
		\draw (26) to (27);
		\draw (28) to (29);
		\draw [dotted] (30.center) to (25);
		\draw [dotted] (31.center) to (25);
		\draw [dotted] (32.center) to (27);
		\draw [dotted] (33.center) to (29);
		\draw [dotted] (34.center) to (29);
		\draw [dotted] (35.center) to (29);
		\draw [dotted] (27) to (28);
		\draw [style=new edge style 1] (46.center) to (47.center);
		\draw (23) to (24);
	\end{pgfonlayer}
\end{tikzpicture}
\end{figure}

  Repeating the above step $k$ times, we obtain a Coxeter system $S'_k$ through the sequence $$S'\to S_1'\to \ldots\to S_k',$$ where $S'_j$ is a twist of $S_{j-1}'$ for $1\leq j\leq k$. Consequently, in $\mathcal{V}_{(W, S'_k)}$, the vertex $t^{\Delta_{0,1}\cdots\Delta_{k-1,k}}$ is incident to $s_0$ with the same edge label $m(s,t)$. Note that during the twisting process from $S'$ to $S'_k$, all vertices and edge incidences remain unchanged except for those involving $t$. By applying the same procedure for every edge, we can attach each edge to the vertex $s_0$. Hence, by repeated application of diagram twisting, we obtain that $\mathcal{V}_{(W, S'_k)}$ is a star diagram. Note that the above-described procedure is independent of the choice of $s_0$ since any two Coxeter generating sets for $W$, whose corresponding Coxeter diagrams are the same, are automorphic. This completes the proof.
\end{proof}

\begin{lemma}\label{Basic subsets match up to blowing up}
Let $W\in \sw$ with generating set $S$, and $S'$ be another Coxeter generating set for $W$. Then, up to blowing up, the basic subsets of $S$ and $S'$ match isomorphically.
\end{lemma}

\begin{proof}
In the Coxeter system $(W, S)$, the subgroups generated by the maximal spherical simplices are precisely the maximal finite standard parabolic subgroups. Then, by Lemma \ref{4.13[MRT07]}, there is a one-to-one correspondence between the subgroups generated by the maximal spherical simplices in $(W, S)$ and $(W, S')$, as no two maximal spherical simplices of $(W, S)$ are conjugate in $W$. Suppose $M$ is a maximal spherical simplex of $(W, S)$, such that $\langle M \rangle \cong D_{2(4k+2)}$ for some $k\geq1$ and $|M|=2$, whereas the corresponding maximal spherical simplex $M'$ has the cardinality $3$. Then, apply blowing up for all such $M$ in $(W, S)$ and obtain a new Coxeter generating set $\widetilde{S}$ from $S$ for $W$ such that the size of the maximal spherical simplices is also preserved under the one-to-one correspondence of the maximal spherical simplices between $(W,\widetilde{S})$ and $(W, S')$.
    
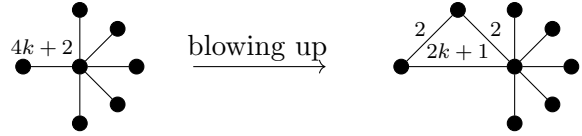
\begin{figure}[ht]
        \centering
\begin{tikzpicture}
	\begin{pgfonlayer}{nodelayer}
		\node [style=new style 0, scale=0.5, color=black] (0) at (1.5, 0) {};
		\node [style=new style 0, scale=0.5, color=black] (1) at (0.75, 0) {};
		\node [style=new style 0, scale=0.5, color=black] (2) at (1.5, 0.75) {};
		\node [style=new style 0, scale=0.5, color=black] (3) at (1.5, -0.75) {};
		\node [style=new style 0, scale=0.5, color=black] (4) at (2, -0.5) {};
		\node [style=new style 0, scale=0.5, color=black] (5) at (2.25, 0) {};
		\node [style=new style 0, scale=0.5, color=black] (6) at (2, 0.5) {};
		\node [style=new style 0, scale=0.5, color=black] (7) at (7.25, 0) {};
		\node [style=new style 0, scale=0.5, color=black] (8) at (5.75, 0) {};
		\node [style=new style 0, scale=0.5, color=black] (9) at (7.25, 0.75) {};
		\node [style=new style 0, scale=0.5, color=black] (10) at (7.25, -0.75) {};
		\node [style=new style 0, scale=0.5, color=black] (11) at (7.75, -0.5) {};
		\node [style=new style 0, scale=0.5, color=black] (12) at (8, 0) {};
		\node [style=new style 0, scale=0.5, color=black] (13) at (7.75, 0.5) {};
		\node [style=none] (14) at (3, 0) {};
		\node [style=none] (15) at (4.75, 0) {};
		\node [style=none] (16) at (3.85, 0.25) {\small{blowing up}};
		\node [style=new style 0, scale=0.5, color=black] (17) at (6.5, 0.75) {};
		\node [style=none] (18) at (1, 0.25) {\tiny{$4k+2$}};
		\node [style=none] (19) at (6.5, 0.2) {\tiny{$2k+1$}};
		\node [style=none] (20) at (6, 0.5) {\tiny{$2$}};
		\node [style=none] (21) at (7, 0.5) {\tiny{$2$}};
	\end{pgfonlayer}
	\begin{pgfonlayer}{edgelayer}
		\draw [style=new edge style 1] (14.center) to (15.center);
		\draw (0) to (5);
		\draw (0) to (6);
		\draw (0) to (2);
		\draw (0) to (3);
		\draw (0) to (4);
		\draw (7) to (10);
		\draw (7) to (11);
		\draw (7) to (12);
		\draw (7) to (13);
		\draw (7) to (9);
		\draw (1) to (0);
		\draw (8) to (7);
		\draw (17) to (7);
		\draw (17) to (8);
	\end{pgfonlayer}
\end{tikzpicture}
   \end{figure}
    
Except for the maximal spherical simplices of type $D_{2(2k+1)}\times \mathbb{Z}_2$, the remaining maximal spherical simplices are indeed basic subsets for both the generating set $\widetilde{S}$ and $S'$. The rest of the basic subsets of $\widetilde S$ and $S'$ have type $D_{2(2k+1)}$, containing inside the maximal spherical simplices of type $D_{2(2k+1)}\times \mathbb{Z}_2$. This provides bijections between the maximal spherical simplices and the basic subsets of $\widetilde{S}$ and $S'$. Therefore, by the Lemma $\ref{Basic Matching Theorem}$, we get that the basic subsets of $\widetilde{S}$ and $S'$ match isomorphically.
\end{proof}

\begin{theorem}\label{rigidity theorem for W(T)}
Let $W \in \sw$ and $S'$ be a Coxeter generating set of $W$. Then, up to diagram twisting and triangle elimination, $\mathcal{V}_{(W, S')}$ is a star diagram.
\end{theorem}

\begin{proof}
Let $\widetilde S$ be as described in Lemma $\ref{Basic subsets match up to blowing up}$. By by Theorem $\ref{Simplex Matching Theorem}$ and Lemma $\ref{Basic subsets match up to blowing up}$, we get that, up to blowing up, $|\widetilde S|=|S'|$ and edge-label multiset of the diagrams $\mathcal{V}_{(W,\widetilde{S})}$ and $\mathcal{V}_{(W, S')}$ agree. Moreover, there is a bijective correspondence between the finite standard parabolic subgroups of the same isomorphism type for the systems $(W,\widetilde S)$ and $(W, S')$. So the only loops in $(W, S')$ are the ones corresponding to the maximal spherical simplex of type $D_{2(2k+1)}\times \mathbb{Z}_2$. If an edge $e'$ in $\mathcal{V}_{(W, S')}$ is evenly labelled, then one of its incident vertices, say $v'$, has degree one. This follows because the maximal spherical simplex $M'$ corresponding to $e'$ is conjugate to a maximal spherical simplex $M$ of $(W, S)$ associated with an edge $e$ where $v'$ corresponds to a vertex $v$ and $\langle M \rangle$ is the unique maximal parabolic subgroup containing $v$.

Consider a maximal spherical simplex $M'$ of type $D_{2(2k+1)}\times \mathbb{Z}_2$. Then the vertex incident to both the edges of labelled $2$ is not adjacent to any other vertices in $\mathcal{V}_{(W, S')}$ since that vertex corresponds to the unique maximal element of the a dihedral system of type $D_{2(4k+2)}$ which has no non-trivial relations with the remaining Coxeter generating elements in $S'$ due to the correspondence of maximal spherical simplices. of $\widetilde{S}$ and $S'$. Hence, up to diagram twisting, we may assume that there is exactly one vertex of degree greater than or equal to $3$ in $M'$.
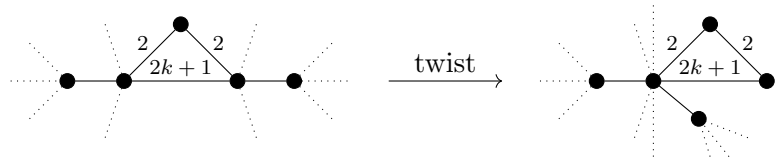
\begin{figure}[ht]
        \centering
\begin{tikzpicture}
	\begin{pgfonlayer}{nodelayer}
		\node [style=new style 0, scale=0.5, color=black] (0) at (0, 0) {};
		\node [style=new style 0, scale=0.5, color=black] (1) at (0.75, 0) {};
		\node [style=new style 0, scale=0.5, color=black] (2) at (2.25, 0) {};
		\node [style=new style 0, scale=0.5, color=black] (3) at (3, 0) {};
		\node [style=new style 0, scale=0.5, color=black] (4) at (1.5, 0.75) {};
		\node [style=none] (5) at (0.5, -0.75) {};
		\node [style=none] (6) at (0.5, 0.75) {};
		\node [style=none] (7) at (2.5, -0.75) {};
		\node [style=none] (8) at (2.5, 0.75) {};
		\node [style=none] (9) at (-0.5, -0.5) {};
		\node [style=none] (10) at (-0.5, 0.5) {};
		\node [style=none] (11) at (3.5, -0.5) {};
		\node [style=none] (12) at (3.5, 0.5) {};
		\node [style=none] (13) at (-0.75, 0) {};
		\node [style=none] (14) at (3.75, 0) {};
		\node [style=none] (15) at (1.5, 0.18) {\tiny{$2k+1$}};
		\node [style=none] (16) at (1, 0.5) {\tiny{$2$}};
		\node [style=none] (17) at (2, 0.5) {\tiny{$2$}};
		\node [style=new style 0, scale=0.5, color=black] (18) at (7, 0) {};
		\node [style=new style 0, scale=0.5, color=black] (19) at (7.75, 0) {};
		\node [style=new style 0, scale=0.5, color=black] (20) at (9.25, 0) {};
		\node [style=new style 0, scale=0.5, color=black] (22) at (8.5, 0.75) {};
		\node [style=none] (23) at (7.5, -0.75) {};
		\node [style=none] (24) at (7.5, 0.75) {};
		\node [style=none] (25) at (7.75, -1) {};
		\node [style=none] (26) at (7.75, 1) {};
		\node [style=none] (27) at (6.5, -0.5) {};
		\node [style=none] (28) at (6.5, 0.5) {};
		\node [style=none] (29) at (8.5, -1) {};
		\node [style=none] (30) at (9, -0.75) {};
		\node [style=none] (31) at (6.25, 0) {};
		\node [style=none] (32) at (8.75, -1) {};
		\node [style=none] (33) at (8.5, 0.18) {\tiny{$2k+1$}};
		\node [style=none] (34) at (8, 0.5) {\tiny{$2$}};
		\node [style=none] (35) at (9, 0.5) {\tiny{$2$}};
		\node [style=none] (36) at (4.25, 0) {};
		\node [style=none] (37) at (5.75, 0) {};
		\node [style=none] (38) at (5, 0.25) {\small{twist}};
		\node [style=new style 0, scale=0.5, color=black] (39) at (8.35, -0.5) {};
	\end{pgfonlayer}
	\begin{pgfonlayer}{edgelayer}
		\draw (0) to (1);
		\draw (1) to (2);
		\draw (2) to (3);
		\draw (1) to (4);
		\draw (4) to (2);
		\draw (18) to (19);
		\draw (19) to (20);
		\draw (19) to (22);
		\draw (22) to (20);
		\draw [style=new edge style 1] (36.center) to (37.center);
		\draw (19) to (39);
		\draw [dotted] (13.center) to (0);
		\draw [dotted] (10.center) to (0);
		\draw [dotted] (9.center) to (0);
		\draw [dotted] (1) to (5.center);
		\draw [dotted] (6.center) to (1);
		\draw [dotted] (8.center) to (2);
		\draw [dotted] (7.center) to (2);
		\draw [dotted] (11.center) to (3);
		\draw [dotted] (3) to (14.center);
		\draw [dotted] (12.center) to (3);
		\draw [dotted] (31.center) to (18);
		\draw [dotted] (27.center) to (18);
		\draw [dotted] (18) to (28.center);
		\draw [dotted] (23.center) to (19);
		\draw [dotted] (24.center) to (19);
		\draw [dotted] (29.center) to (39);
		\draw [dotted] (32.center) to (39);
		\draw [dotted] (30.center) to (39);
		\draw [dotted] (26.center) to (19);
		\draw [dotted] (25.center) to (19);
	\end{pgfonlayer}
\end{tikzpicture}
\end{figure}

By using triangle elimination, from $S'$, we obtain a Coxeter generating set $\widetilde{S'}$ whose corresponding finite diagram is a tree such that all the even-labelled edges are incident to some vertex of degree one in $\mathcal{V}_{(W,\widetilde{S'})}$. Hence, Lemma \ref{structure of tree diagrams for groups in W(T)} can be applied to $\widetilde{S'}$, which gives us a star diagram $\widetilde{\mathcal{V}}$ from $\mathcal{V}_{(W,\widetilde{S}')}$ using diagram twisting. So, $\widetilde{\mathcal{V}} \cong \mathcal{V}_{(W, S)}$ and there is an automorphism $\varphi \in \Aut(W)$ which takes $S$ to the Coxeter generating set corresponding to the vertex set of $\widetilde{\mathcal{V}}$.
\end{proof}

\begin{corollary}
If $W\in \sw$, then $W$ is rigid up to diagram twisting and blowing up.
\end{corollary}

\begin{corollary}
Let $W\in \sw$. Then $W$ is rigid if and only if $W$ is strongly even.    
\end{corollary}

\begin{remark}
By Theorem $\ref{rigidity theorem for W(T)}$, the isomorphism problem for the family $\sw$ of Coxeter groups is solved.
\end{remark}

\begin{ack}
The authors acknowledge the support of National Institute of Science Education and Research (NISER), Bhubaneswar, and Homi Bhabha National Institute (HBNI), Mumbai. The second named author acknowledges partial support from ANRF via the grants SRG/2023/001556/PMS and ANRF/ARGM/2025/001758/MTR.
\end{ack}

\medskip

\textbf{Addresses}: Arijit Mahato, Tushar Kanta Naik, A Rameswar Patro.
\smallskip
\begin{itemize}
\item[1.] School of Mathematical Sciences, 
\newline National Institute of Science Education and Research (NISER), Bhubaneswar,
\newline P. O. Jatni, Khurda, Odisha - 752050, India.
\smallskip
\item[2.] Homi Bhabha National Institute (HBNI), 
\newline Training School Complex, Anushakti Nagar,  Mumbai 400094, India.
\end{itemize}
\smallskip

\textbf{E-mails}: arijit.mahato@niser.ac.in, tushar@niser.ac.in, arameswar.patro@niser.ac.in.
\end{document}